\documentclass[a4paper]{article}
\usepackage{hyperref}
\usepackage{graphicx}
\usepackage{mathrsfs}
\usepackage{enumerate} 
\usepackage{amsxtra,amssymb,latexsym, amscd,amsthm}
\usepackage{indentfirst}
\usepackage{color}
\usepackage[utf8]{inputenc}
\usepackage[mathscr]{eucal}
\usepackage{amsfonts}
\usepackage{graphics}
\usepackage{multirow}
\usepackage{array}
\usepackage{subfigure}
\usepackage{cite}
\usepackage{wrapfig}
\usepackage{pgf,tikz,pgfplots}
\pgfplotsset{compat=1.14}

\usetikzlibrary{arrows}
\newcommand{\footremember}[2]{%
    \footnote{#2}
    \newcounter{#1}
    \setcounter{#1}{\value{footnote}}%
}
\newcommand{\footrecall}[1]{%
    \footnotemark[\value{#1}]%
} 

\def\R{{\mathbb R}}
\def\N{{\mathbb N}}

\def\cM{\mathcal M}

\DeclareMathOperator{\rank}{rank}

\DeclareMathOperator{\trace}{trace}

\newtheorem{lemma}{\bf Lemma}[section]

\newtheorem{algorithm}{\bf Algorithm}[section]
\newtheorem{example}{\bf Example}[section]
\newtheorem{theorem}{\bf Theorem}[section]
\newtheorem{proposition}{\bf Proposition}[section]
\newtheorem{corollary}{\bf Corollary}[section]

\newtheorem{remark}{\bf Remark}[section]
\providecommand{\keywords}[1]
{
  \small	
  \textbf{\textbf{Keywords:}} #1
}
\begin{document}
\definecolor{qqzzff}{rgb}{0,0.6,1}
\definecolor{ududff}{rgb}{0.30196078431372547,0.30196078431372547,1}
\definecolor{xdxdff}{rgb}{0.49019607843137253,0.49019607843137253,1}
\definecolor{ffzzqq}{rgb}{1,0.6,0}
\definecolor{qqzzqq}{rgb}{0,0.6,0}
\definecolor{ffqqqq}{rgb}{1,0,0}
\definecolor{uuuuuu}{rgb}{0.26666666666666666,0.26666666666666666,0.26666666666666666}
\newcommand{\vi}[1]{\textcolor{blue}{#1}}
\newif\ifcomment
\commentfalse
\commenttrue
\newcommand{\comment}[3]{%
\ifcomment%
	{\color{#1}\bfseries\sffamily#3%
	}%
	\marginpar{\textcolor{#1}{\hspace{3em}\bfseries\sffamily #2}}%
	\else%
	\fi%
}
\newcommand{\victor}[1]{
	\comment{blue}{V}{#1}
}
\title{Positivity certificates and polynomial optimization on non-compact semialgebraic sets}
\author{%
Ngoc Hoang Anh Mai\footremember{1}{CNRS; LAAS; 7 avenue du Colonel Roche, F-31400 Toulouse; France.}, %
  Jean-Bernard Lasserre\footrecall{1} \footremember{2}{Universit\'e de Toulouse; LAAS; F-31400 Toulouse, France.}, %
   Victor Magron\footrecall{1} %
  }
\maketitle
\begin{abstract}
In a first contribution, we revisit two certificates of positivity on (possibly non-compact) basic semialgebraic sets due to 
Putinar and Vasilescu [Comptes Rendus de l'Acad\'emie des Sciences-Series I-Mathematics, 328(6) (1999) pp. 495-499]. 
We use Jacobi's technique from [Mathematische Zeitschrift, 
237(2) (2001) pp. 259-273] to provide an alternative proof with an {\em effective 
degree bound} on the sums of squares multipliers in such certificates. 
As a consequence, it allows one to define a hierarchy of semidefinite relaxations for a general polynomial optimization problem. Convergence of this hierarchy to a neighborhood of the optimal value as well as strong duality and analysis are guaranteed. 
In a second contribution, we introduce a new numerical method for solving systems of 
polynomial inequalities and equalities with possibly uncountably many solutions. As a bonus, one may 
apply this method to obtain approximate global optimizers in polynomial optimization.
\end{abstract}
\tableofcontents
\section{Introduction}
This paper is concerned with polynomial optimization on non-compact semialgebraic sets. Its spirit and main motivation is to voluntarily avoid the {\em big-ball trick} which reduces the problem to the compact case. 
The big-ball ``trick" is to simply assume that the global minimum is attained in some {\em a priori known} ball $B_M$ centered at zero of radius $M>0$ potentially large. 
Therefore, by adding this additional constraint to the definition of the feasible set, one is back to the compact case.

Why?  This ``trick" has definitely some merit since in some practical applications such an $M$ can be sometimes determined with {\em ad-hoc} arguments. 
However, it is not satisfactory from a mathematical point of view. Indeed after one has found a minimizer $x^\star\in B_M$, one is still left with the question: {\em Is really $x^\star$ a global minimizer? 
Was $M$ chosen sufficiently large?} In other words, in doing so one does not obtain an certificate that $x^\star$ is a global minimizer. 
As we will see, the challenge is to {\em adapt} some certificates of positivity on non-compact sets already available in the literature, to turn them
into a practical algorithm.

\paragraph{Background.} Deciding nonnegativity of a polynomial is an important and attractive problem throughout history of the development of real algebraic geometry.  
In his famous and seminal work \cite{hilbert1888darstellung}, Hilbert characterized all cases where nonnegative polynomials are sums of squares (SOS) of polynomials and later
Blekherman showed in~\cite{blekherman2006there} that there are significantly more nonnegative polynomials than SOS. 
In 1927, Artin proved in~\cite{artin1927zerlegung} that every nonnegative polynomial can be decomposed as a sum
of squares of rational functions, thereby solving Hilbert's 17th probem.
Namely, $f$ is nonnegative if and only if 
$\sigma_D f = \sigma_N$ for some SOS polynomials $\sigma_N$ and $\sigma_D$.
Later, certificates of positivity on a general semialgebraic set have been proposed by 
Stengle \cite{stengle1974nullstellensatz} (see also Krivine \cite{krivine1964anneaux}). 
A \emph{basic semialgebraic} set $S(g,h)$ can be written as
\begin{equation}
    \label{eq:S(g,h)}
S(g,h)\,:=\,\{\,x \in \R^n :\: g_j(x)\geq0\,,\:j=1,\ldots,m;\:h_t(x)=0\,,\:t=1,\ldots,l\,\} \,,
\end{equation}
where $n$ is the dimension of the ambient space and $(g_j, h_{t})$ are polynomials.
Stengle and Krivine rely on a tool from real algebraic geometry called \emph{preordering} 
\begin{equation}
\label{eq-preordering}
P\left( {g,h} \right) := \left\{ {\sum\limits_{\alpha  \in {{\{ 0,1\} }^m}} {{\sigma _\alpha }g_1^{{\alpha _1}} \ldots g_m^{{\alpha _m}}}  + \sum\limits_{t = 1}^l {{\phi _t}{h_t}} :\,  {\sigma _\alpha } \in \Sigma \left[ x \right],\ {\phi _t} \in \R\left[ x \right]} \right\}
\,,
\end{equation}
associated with the polynomials $(g_j,h_t)$. Here $\R[x]$ denotes the ring of real polynomials and $\Sigma[x] \subset \R[x]$ stands for the set of SOS polynomials. Krivine-Stengle's Positivstellensatz (or  {\em representation}) states that 
\begin{eqnarray}
\label{stengle-1}
f \geq 0\text{ on }S( {g,h} ) &\Leftrightarrow& \exists {q_1},{q_2} \in P( {g,h} ),\ s \in \N:\ {q_1}f = {f^{2s}} + {q_2}\\
\label{stengle-2}
f > 0 \text{ on }S({g,h} ) &\Leftrightarrow& \exists {q_1},{q_2} \in P( {g,h} ):\, {q_1}f = 1 + {q_2}\,.
\end{eqnarray}
Notice that the above representations involve a multiplier $q_1$ for $f$ as well as cross-products of the $g_j$'s in \eqref{eq-preordering}. 
In 1993, Putinar \cite{putinar1993positive} refined a result of Schm\"udgen \cite{schmudgen1991thek} for certificates of positivity on a basic semi algebraic set \eqref{eq:S(g,h)} assumed to be compact plus an {\em Archimedean} assumption, described below. It avoids a multiplier for $f$ and no cross-product of the $g_j$'s.
Namely, Putinar's Positivstellensatz states that $f$ is positive on $S(g,h)$ if $f$ belongs to the set
\begin{equation}
\label{cert-put}
Q\left( {g,h} \right): = \left\{ {\sigma_0 + \sum\limits_{j = 1}^m {{\sigma _j}{g_j}}  + \sum\limits_{t = 1}^l {{\phi _t}{h_t}} :\,{\sigma _j} \in \Sigma \left[ x \right],\,{\phi _t} \in \R\left[ x \right]} \right\} \,.
\end{equation}
The set $Q(g,h)$ is called the {\em quadratic module} associated with the polynomials $(g_j,h_t)$.
\paragraph{SOS for optimization.}
More recently and since the pioneer works of Lasserre \cite{lasserre2001global} and Parrilo \cite{phdParrilo}, SOS-based certificates of nonnegativity have now become a powerful tool in polynomial optimization and control.
 In the unconstrained case,  let $f^\star := \inf_{x \in \R^n} f(x)$. If $f-f^\star$ ($\geq0$ on $\R^n$) is an SOS polynomial
 then
$f^\star$ can be obtained by solving a single semidefinite program (SDP). However in general $f-f^\star$ is an SOS of rational functions, which yields:
\begin{equation}
    \label{eq:test}
f^\star\,=\,\displaystyle\sup_{\lambda,\sigma_N,\sigma_D}\,
\{\lambda :  \,\sigma_D\,(f-\lambda)\,=\,\sigma_N\,;\quad \sigma_N,\sigma_D\,\in\,\Sigma[x]\,\}\,
\end{equation}
(plus a normalizing constraint for $\sigma_N$ to avoid the trivial solution $(+\infty,0,0)$). By fixing in advance a bound $d$ on the degree of $\sigma_D$, one may solve 
\eqref{eq:test} by SDP combined with bisection search on $\lambda$ and increase of $d$ when no solution exists.
In the constrained case, let $S(g,h)$ in \eqref{eq:S(g,h)} be compact and assume with no loss of generality that the so-called {\em Archimedean} assumption holds, namely that $L-\| x\|^2_2$ belongs to $Q\left( {g,h} \right)$ for some $L>0$. This can be automatically ensured by  setting $g_m(x)=L-\| x\|^2_2$. 
Under this assumption, the second author provides in \cite{lasserre2001global} a so-called \emph{moment-SOS hierarchy} of semidefinite relaxations based on Putinar's representation, yielding a non-decreasing sequence of lower bounds on ${f^\star} := \min_{x \in S({g,h})}f( x)$, which converges to $f^\star$. 
Generically convergence is finite \cite{nie2014optimality} and  a numerical procedure from \cite{henrion2005detecting} allows one to extract global minimizers from an optimal solution 
of the (exact) semidefinite relaxation in the hierarchy. It relies on the flat extension condition of Curto and Fialkow 
\cite{curto2005truncated,laurent2005revisiting}. In the above-mentioned frameworks, compactness of $S(g,h)$ is crucial.
\paragraph{Related works on SOS approximations of nonnegative polynomials.}
Blekherman's result \cite{blekherman2006there} states that for a fixed degree, the cone of nonnegative polynomials is way larger than the cone of SOS polynomials.
This is in contrast with the denseness result from \cite[Theorem 5]{Berg80}, which establishes that the cone of SOS polynomials is dense in the space of polynomials being nonnegative on $[-1, 1]$, for the $l_1$-norm of coefficients, defined by $\|f\|_1 = \sum_{\alpha} |f_\alpha|$ (whenever one writes $f = \sum_{\alpha} f_\alpha x^{\alpha}$ in the standard canonical basis of monomials).
Other denseness results from \cite{lasserre2007sum,lasserre2007sos} are based on perturbations of nonnegative polynomials to obtain SOS certificates. 
In \cite{lasserre2007sum}, Lasserre states that any given nonnegative polynomial $f$ can be approximated by a sequence $(f_\varepsilon)_{\varepsilon}$ of SOS  polynomials given by
\[{f_\varepsilon } := f + \varepsilon \sum\limits_{k = 0}^{{r_\varepsilon }} {\sum\limits_{j = 1}^n {x_j^{2k}/( {k!} )} } , \ \varepsilon>0 \,,\]
for some $r_\varepsilon \in \N$, so that $\|f_\varepsilon - f\|_1 \to 0$, as $\varepsilon \downarrow 0$.  
Similarly, Lasserre and  Netzer prove in \cite{lasserre2007sos} that every polynomial $f$ being nonnegative on the unit box $[0, 1]^n$ can be approximated in $l_1$-norm by a sequence of SOS 
\[{f_{r \varepsilon}} := f + \varepsilon \left( {1 + \sum\limits_{j = 1}^n {x_j^{2r}} } \right), \ \varepsilon>0  \,.\]
Provided that $r$ is large enough, one has $\|f_{r \varepsilon} - f \|_1 \to 0$, as $\varepsilon \downarrow 0$.
Jibetean and Laurent \cite{jibetean2005semidefinite} compute tight upper bounds for the unconstrained polynomial optimization problem $f^\star:=\min_{x\in\R^n}f(x)$ based on the perturbed problem $f^\star_\varepsilon:=\inf_{x\in\R^n}f(x)+\varepsilon\sum_{j = 1}^n {x_j^{2d + 2}} $ and SDP relaxations over the gradient ideal (see also \cite{nie2006minimizing}). 
Ahmadi and Hall \cite{ahmadi2019construction} provide some positivity certificates on compact domains relying on Artin's and Reznick's results via homogenization. These certificates also require a perturbation and additional dual variables associated with the constraints as well as a homogenization variable. In the end, they build up a convergent hierarchy where they only have to check that some homogeneous polynomial has positive coefficients.
Muramatsu et al. \cite{muramatsu2016perturbed} consider another certificate of positive polynomials on a compact set with a small perturbation. Despite the lack of convergence guarantees, their computational experiments show that their SDP relaxations have often smaller dimension than the ones from Lasserre's hierarchy \cite{lasserre2001global}. This is obtained by restricting each SOS multiplier involved in a Putinar's representation to be supported in a dilation of the associated constraint's support.
Besides their theoretical aspects, these approximation results allow one to interpret some paradoxical behaviors (due to numerical roundoff errors) observed while relying on SDP relaxations for polynomial optimization.
Such behavior occurs for instance while extracting the minimizers of Motzkin's polynomial $f = ( {x_1^2 + x_2^2-1})x_1^2x_2^2+1/27$ with the algorithmic procedure from \cite{henrion2005detecting}. 
Motzkin's polynomial is globally nonnegative but does not belong to the SOS cone. 
However, the perturbed polynomial $\tilde f = f + \varepsilon (1 + x_1^6 + x_2^6)$ is an SOS for small $\varepsilon > 0$.
This implies that an SDP solver can find an approximation of the optimal value of $f$ for a sufficiently high order of relaxation, and that one can extract the global minimizers of $f$. 
In fact ~\cite{lasserre2019sdp}, an SDP solver performs ``robust optimization'' in the following sense: instead of solving the original optimization problem with nominal criterion $f$, the solver considers a perturbed criterion which lies in a ball of small radius $\varepsilon$ and center $f$.
In \cite{navascues2013paradox}, the authors explain a similar paradox occurring in a noncommutative setting. \\
As shown in \cite{magron2018exact}, the user can also introduce perturbations to compensate the numerical uncertainties added by the solver. 
This perturbation/compensation scheme is the main ingredient of the hybrid numeric-symbolic algorithm from \cite{magron2018exact}, designed to compute exact rational SOS decompositions for polynomials lying in the interior of the SOS cone.
\paragraph{The non-compact case.} There have been several attempts to provide a hierarchy of semidefinite relaxations when $S(g,h)$ is not compact. In \cite{jeyakumar2014polynomial} the authors consider polynomial optimization problems with non-compact set in the special case where $S( {g \cup \{ {c - f} \},h} )$ satisfies the Archimedian assumption. 
Later on, Dickinson and Povh \cite{dickinson2015extension} obtain a certificate for homogeneous polynomials positive on the intersection of the nonnegative orthant with  a basic semialgebraic cone.  
This latter result generalizes  P\'olya \cite{polya1928positive} and Reznick's Positivstellensatz \cite{reznick1995uniform}. 
P\'olya's result states that one can always multiply a homogeneous  polynomial positive on the nonnegative orthant by some power of $(x_1 + \dots + x_n)$ to obtain a polynomial with nonnegative coefficients.
Reznick proves that any positive definite form can be multiplied by a certain power of $\|x\|_2^2$ to become a SOS.
For this specific class of nonnegative polynomials, Reznick's result provides a suitable decomposition into SOS of rational functions, which can be practically computed via SDP. 
%
An interesting related result is the Positivstellensatz \cite{putinar1999solving} of Putinar and Vasilescu. Let us define 
\[S(g)\,:=\,\{x:\:g_j(x)\geq0,\: j=1,\ldots,m\} \ \text{and}
\
Q\left( {g} \right)\,: =\, \left\{ {\sigma_0+ \sum\limits_{j = 1}^m {{\sigma _j}{g_j}}  :\;{\sigma _j} \in \Sigma [ x ]} \right\}\,.\]
\begin{theorem}(Putinar-Vasilescu\cite[Corollary 4.3 and 4.4]{putinar1999solving})\label{theo:Putinar.Vasilescu}
Let $\theta\in\R[x]$ be the quadratic polynomial $x\mapsto \theta(x):= 1 + \|x\|_2^2$, and
denote by $\tilde p \in \R[ {x,{x_{n + 1}}} ]$ the homogeneous polynomial associated with $p\in \R[ {x} ]$, defined by $x\mapsto \tilde p(x): = x_{n + 1}^{\deg ( p )}p( {x/{x_{n + 1}}} )$.
\begin{enumerate} 
\item Let $f\in \R[ {x} ]$ such that $\tilde f>0$ on ${\R^{n + 1}}\backslash \{ 0 \}$. 
Then ${\theta ^k}f \in \Sigma [ x ]$ for some $k\in\N$.
\item Let $f,g_1,\dots,g_m \in \R[ {x} ]$ satisfy  the following two conditions:
\begin{enumerate}
\item $f = {f_0} + {f_1}$ such that $\deg ({f_0}) < \deg ( {f_1})$ and ${{\tilde f}_1} > 0$ on ${\R^{n + 1}}\backslash \{ 0 \}$;
\item $f > 0$ on $S(g)$.
\end{enumerate}
Then ${\theta ^{2k}}f \in Q(g)$ for some $k\in \N$.
\end{enumerate}
\end{theorem}
 
As a consequence, they also obtain:
\begin{corollary}(Putinar-Vasilescu \cite[Final remark 2]{putinar1999solving})  
\label{coro:consequence.Purtina}
Let $\theta := 1 + \|x\|_2^2$.
\begin{enumerate}
\item Let $f \in \R[ x ]_{2d}$ be such that $f\ge 0$ on $\R^n$. 
Then for all $\varepsilon>0$, there exists $k\in \N$ such that ${\theta ^k}( {f + \varepsilon {\theta ^d}} ) \in \Sigma [ x]$.
\item Let $f \in \R[ x ]$ such that $f\ge 0$ on $S( g )$. Let $d\in \N$ such that $2d > \deg ( f )$. 
Then for all $\varepsilon>0$, there exists $k\in \N$ such that ${\theta ^{2k}}( {f + \varepsilon {\theta ^d}} ) \in Q ( g )$.
\end{enumerate}
\end{corollary}
Marshall \cite[Corollary 4.3]{marshall2003approximating} states a slightly more general result but with no explicit $d$, and Schweighofer \cite[Corollary 6.3]{schweighofer2003iterated} provides a new algebraic proof of Marshall's result.
To summarize, for every polynomial $f$ nonnegative on a general basic semialgebraic set $S(g)$, one obtains the following representation result: 
for a given $\varepsilon>0$, there exist a nonnegative integer $k$ and SOS polynomials $\sigma_0, \sigma_1, \dots, \sigma_m$, such that
\begin{equation}\label{eq:rational.rep}
f + \varepsilon \theta^d  = \frac{{{\sigma _0} + {\sigma _1}{g_1} + \dots + {\sigma _m}{g_m}}}{{{\theta ^k}}} \,.
\end{equation}
Although this representation is theoretically attractive, their previous proofs are not constructive and do not provide any explicit algorithm, especially in polynomial optimization. 
The underlying reason is that there is no degree bounds on the SOS weights  $\sigma_0,\dots,\sigma_m$. 
Therefore, checking membership of  ${\theta ^k}( {f + \varepsilon {\theta ^d}} )$ in $Q(g )$ is still a challenge.
If one fixes a degree $k$ which is not large enough then one would have to increase the degree of the $\sigma_j$'s forever without getting an answer.
This restriction comes from the proof techniques used by Putinar and Vasilescu, Marshall, and Schweighofer used in \cite{putinar1999solving}, \cite{marshall2003approximating} and \cite{schweighofer2003iterated}, respectively. 
The main idea of Putinar and Vasilescu is to localize the ring of polynomials by allowing inverses of ${\| {( {x,{x_{n + 1}}} )} \|^2_2}$ and
then to use  Cassier's technique  for separating two convex cones of rational functions. 
Marshall proves Corollary \ref{coro:consequence.Purtina} by applying the generated Jacobi-Prestel criterion.
Schweighofer's proof for representation \eqref{eq:rational.rep} is based on the relationship between the subring of bounded elements and the subring of arithmetically bounded elements, together with  Monnier's conjecture.
From their proofs it is possible but quite difficult to find  degree bounds for the SOS weights $\sigma_j$ with respect to the input polynomial data $f,g_j$.

\paragraph{\textbf{Contribution}.} 
As already mentioned, our approach is to 
treat the non-compact case frontly and avoid the big-ball trick. Our contribution is threefold:

{\bf I.} In Section~\ref{sec:representations} we first provide
an alternative (and simpler) proof of \eqref{eq:rational.rep} and {\em with an explicit degree bound} on the SOS multipliers; this is crucial as it has immediate implications on the algorithmic side.
 More precisely, the degrees of SOS multipliers $\sigma_j$ 
are bounded above by $k + d - \lceil {\deg ( {{g_j}} )/2} \rceil $.
Our proof for \eqref{eq:rational.rep} relies on Jacobi's technique in the proof of \cite[Theorem 7]{jacobi2001representation}. First, one transforms the initial polynomials to homogeneous forms, then one relies on Putinar's Positivstellensatz for the compact case, and finally one transforms back the obtained forms to dehomogenized polynomials.  
As a consequence, with $\varepsilon>0$ fixed, arbitrary, this degree bound allows us to provide hierarchies $(\rho _k^i(\varepsilon))_{k\in\N}$, $i=1,2,3$ for unconstrained polynomial optimization ($m=0$ and $i=1$, see Section~\ref{sec:unconstrained_hierarchy}) as well as for constrained  polynomial optimization ($m\geq 1$ and $i=2,3$, see Section~\ref{sec:constrained_hierarchy}). 
Computing each $\rho _k^i(\varepsilon)$ boils down to solving a single SDP, with strong duality property. 
For $k$ sufficiently large, $\rho _k^i(\varepsilon)$ becomes an upper bound for the optimal value $f^\star$ of the corresponding polynomial optimization problem (POP) $\min_{x\in S(g)}f(x)$.
If this problem has an optimal solution $x^\star$, the gap between $\rho _k^i(\varepsilon)$ and $f^\star$ is at most $\varepsilon \theta (x^\star)^d$. The related convergence rates are also analyzed in these sections.

{\bf II.} In Section \ref{sec:optimizers}, we provide a new algorithm to find a feasible solution
in the set $S(g,h)$ defined in \eqref{eq:S(g,h)}. 
The idea is to include appropriate additional spherical equality constraints $\varphi_t:={\xi _t} - \| {x - {a_t}} \|_2^2$, $t = 0,\dots,n$, in $S( {g,h})$ so that the system $S( {g,h\cup \{ {{\varphi _0},\dots,{\varphi _n}} \}} )$ has a unique real solution. The nonnegative reals $({\xi _t})_{t=0,\dots,n}$ are computed with an adequate moment-SOS hierarchy. 
Moreover, this solution can be extracted  by 
checking whether some (moment) matrix satisfies a flat extension condition. 

{\bf III.} Finally we use this method to approximate a global minimizer of $f$ on $S(g,h)$.
Namely, we fix $\varepsilon>0$ small and find a point in $S( {g\cup \{ {\rho _k^i(\varepsilon) - f} \},h } )$. 
This procedure works generically, no matter if the set of minimizers is infinite. This is in deep contrast with the extraction procedure
of \cite{henrion2005detecting} (via some flat extension condition) which works only for finite solution sets.
Assuming that the set of solutions is finite, one may compare our algorithm with the procedure from \cite{henrion2005detecting} as follows.
On the one hand, the latter extraction procedure provides global optimizers, provided that one has solved an SDP-relaxation with sufficiently large ``$k$" (so as to get an appropriate rank condition). 
On the other hand,  our algorithm 
that adds spherical equality constraints ``divides" the problem into $n+1$ SDP relaxations with additional constraints 
but with smaller order ``$k$'' (which is {\em the} crucial parameter for the SDP solvers).
Numerical examples are provided in Section \ref{sec:benchs}
to illustrate the difference between these two strategies.

For clarity of exposition, most proofs are postponed in Section \ref{sec:Appendix}.
\paragraph{\textbf{Comparison to other methods for solving POPs on non-compact semialgebraic sets}.}
We consider the general POP ${f^\star} = \inf \{ {f( x ):\, x \in S( {g,h} )} \}$ where $S( {g,h})$ are unbounded.
\begin{enumerate}
\item In \cite{jeyakumar2014polynomial, jeyakumar2016solving}, Jeyakumar et al. solve a POPs on an unbounded semialgebraic set after checking the coercivity of $f$ on $S( {g,h})$ and as well as the  Archimedeanness of $S( {g \cup \{ {c - f} \},h})$ for some $c\ge f(\bar x)$ with $\bar x \in S(g,h)$. 
In particular, \cite{jeyakumar2016solving} uses the
polynomial optimization solver SparsePOP developed by Waki \cite{waki2008algorithm} which exploits a structured sparsity of $f$, $g$ and $h$. However checking these
conditions can be difficult.
Our method which solves SDPs for the hierarchy $(\rho _k^i( \varepsilon  ))_{k\in\N}$ avoids checking coercivity and Archimedean assumptions. 
\item Demmel et al. \cite{demmel2007representations, nie2006minimizing} provide two representations of polynomials positive (resp. nonnegative) on $S( {g,h})$ for solving POPs on unbounded domains. 
They state that $f$ can be represented as an SOS of polynomials modulo the KKT ideal on $S(g,h)$ if the minimal value of $f$ on $S(g,h)$ is attained at some KKT point and assuming that one of the following conditions holds:
\begin{enumerate}
\item $f>0$ on $S(g,h)$;
\item $f\ge 0$ on $S(g,h)$ and the KKT ideal is radical.
\end{enumerate}
This method is restricted to the case of global minimums satisfying KKT condition and testing if $f$ belongs to the related KKT preorder requires  a large number of SOS multipliers. Moreover, checking the radical property of the KKT ideal is difficult in general. 
Our method goes beyond these restrictions by only testing membership of the  perturbation of $\theta^k f$ in the truncation of $Q( {g,h})$, even if the KKT condition is not satisfied. Reader may have a look at Ha and Pham \cite[\S 3.3]{vui2010representations} with the same comparison as to Demmel et al. \cite{demmel2007representations}.
\item Schweighofer \cite{schweighofer2006global} extends Nie et al. \cite{nie2006minimizing} ($S(g,h)=\R^n$) to the case that $f$ is bounded from below but does not necessarily attain a minimum. Schweighofer's gradient tentacles method replaces the gradient variety by larger semialgebraic sets. Here we 
assume that $f$ attains its minimum and compute an approximation of the  optimal value $f^\star$ as well as an approximation of some minimum  $x^\star$.
\item Greuet et al. \cite{greuet2014probabilistic} provide a probabilistic algorithm for solving POP on a real algebraic set $S(g,h)=V(h)$. 
They can extract a solution under the following assumptions: $\langle h \rangle $ is radical, $V( h )$ is equidimensional of dimension $d > 0$ and $V( h )$ has finitely many singular points. 
These conditions are also difficult to check in practice and are not required for our method.
\end{enumerate}
\section{Notation, definitions and preliminaries}
\label{sec:Preliminaries}
In this section, we introduce notation and basic facts about polynomial optimization and the moment-sums-of-squares (moment-SOS) hierarchy.
With $x = (x_1,\dots,x_n)$, let $\R[x]$ stands for the ring of real polynomials and let 
$\Sigma[x]\subset\R[x]$ be the subset of SOS polynomials.
Let us note $\R[x]_t$ and $\Sigma[x]_t$ the respective restrictions of these two sets to polynomials of degree at most $t$ and $2t$. 
Given $\alpha = (\alpha_1,\dots,\alpha_n) \in \N^n$, we note $|\alpha| := \alpha_1 + \dots + \alpha_n$.
Let $(x^\alpha)_{\alpha\in\N^n}$ 
be the canonical basis of monomials for  $\R[x]$ (ordered according to the graded lexicographic order) and 
$v_t(x)$ be the vector of monomials up to degree $t$, with length $s(t) = {n+t\choose n}$.
A polynomial $f\in\R[x]_t$ is written as  
$f(x)\,=\,\sum_{| \alpha | \leq t} f_\alpha\,x^\alpha\,=\,\mathbf{f}^Tv_d(x)$, 
where $\mathbf{f}=(f_\alpha)\in\R^{s(t)}$ is its vector of coefficients in the canonical basis.
A polynomial $f$ is \emph{homogeneous} of degree $t$ if
$f(\lambda x)=\lambda^t f(x)$ for all $x\in\R^n$ and each $\lambda\in\R$. Equivalently, a homogeneous polynomial can be written as 
$f = \sum_{| \alpha | = t} {f_\alpha x^\alpha }$. The
{\em degree-$d$ homogenization} $\tilde f$ of $f\in\R[x]_d$ is a  homogeneous polynomial of degree $d$ in $n+1$ variables, defined by 
$\tilde f ( x,x_{n + 1}) = x_{n + 1}^{d} f ( x/x_{n + 1} )$.
A {\em positive form} is a nonnegative homogeneous polynomial which is positive everywhere except at the origin. 
\paragraph{Moment and localizing matrix.} 
For a given real-valued sequence $y=(y_\alpha)_{\alpha\in\N^n}$, let us define the Riesz linear functional $L_y:\R[ x ] \to \R$ by:
\[f\mapsto {L_y}( f ) := \sum_{\alpha} f_\alpha y_\alpha,\quad \forall f\in\R[x]\,.\]  
We say that a real infinite (resp. finite) sequence $( y_\alpha)_{\alpha  \in \N^n}$ (resp. $( y_\alpha)_{\alpha  \in \N^n_t}$) has a representing measure if there exists a finite Borel measure $\mu$ such that $y_\alpha  = \int_{\R^n} {x^\alpha d\mu(x)}$ is satisfied for every $\alpha  \in {\N^n}$ (resp. $\alpha  \in {\N^n_t}$). In this case, $( y_\alpha)_{\alpha  \in \N^n}$ is called be the moment sequence of $\mu$. 
Next, given $y=(y_\alpha)_{\alpha  \in \N^n}$ and $d\in \N^*$,
the moment matrix $M_d(y)$ of degree $d$ associated to $y$  is the real symmetric matrix of size $s(d)$ defined by $M_d(y) := ( y_{\alpha  + \beta })_{\alpha,\beta\in \N^n_d} $. 
Let  $g = \sum_{\gamma} g_\gamma x^\gamma  \in \R[x]$. 
The localizing matrix $M_d(gy)$ of degree $d$ associated with $y$ and $g$ is the real symmetric matrix of the size $s(d)$ given by $M_d(gy) = (\sum_\gamma  {{g_\gamma }{y_{\gamma  + \alpha  + \beta }}})_{\alpha, \beta\in \N^n_d}$.
We next recall an important result of \cite{putinar1993positive}, which is crucial for convergence of the moment-SOS hierarchy provided in the sequel. 
\begin{theorem}
\label{theo:Putinar's.Positivstellensatz2}(Putinar \cite{putinar1993positive})
Given $g_1,\ldots,g_m, h_1,\ldots,h_l \in \R[x]$, let
$S(g,h)\subset\R^n$ be as in \eqref{eq:S(g,h)} and $Q(g,h)\subset\R[x]$ be as in \eqref{cert-put}.
Assume that there exists $L > 0$ such that $L - \|x\|_2^2 \in Q( g,h )$ (Archimedean condition). 

\noindent (i) If a polynomial $f\in \R[ x ]$ is  positive on $S( g,h )$ then
$f\in Q(g,h)$.

\noindent(ii) A real sequence $y=(y_\alpha)_{\alpha\in\N^n}$ has a representing measure $\mu$ on $S(g,h)$ 
if $M_d(y)\succeq0$, $M_d(g_j\,y)\succeq0$, $j=1,\ldots,m$, and
$M_d(h_t\,y)=0$, $t=1,\ldots,l$, for all $d\in\N$.
\end{theorem}

\paragraph{The moment-SOS hierarchy.}
%
Let $S(g,h)\subset\R^n$ be as in \eqref{eq:S(g,h)} and $Q(g,h)\subset\R[x]$ be as in \eqref{cert-put} and assume that $Q(g,h)$ is Archimedean.
\if{
$S( g)\, : = \{ {x \in {\R^n};\ {g_j}( x) \ge 0} \}$
 associated with $g = \{ {{g_1},\dots,{g_m}} \} \subset \R[ x ]$. 
For every $d\in \N$, let us consider the truncated quadratic module 
\[
{{ Q}_d}\left( g \right): = \left\{ {{\sigma _0} + \sum\limits_{j = 1}^m {{\sigma _j}{g_j}} ;\ {\sigma _0} \in \Sigma {{\left[ x \right]}_d},\ {\sigma _j} \in \Sigma {{\left[ x \right]}_{d - \left\lceil {\deg \left( {{g_j}} \right)/2} \right\rceil }}} \right\} \,.
\]
}\fi
Consider the POP:
\begin{equation}
    \label{def-pb}
  f^\star := \min \{ {f(x):\ x \in S(g,h)}\}\,,
\end{equation}
known to be challenging as it is NP hard in general~\cite{laurent2009sums}.
One may rewrite \eqref{def-pb} as: 
\begin{equation}
\label{eq:def-dual}
f^\star=\sup \{ {\lambda\in \R :\ f - \lambda  > 0\text{ on }S(g,h)} \}\,,
\end{equation}
and by invoking Theorem \ref{theo:Putinar's.Positivstellensatz2}(i), $f^\star=\sup\{ {\lambda \in \R :\ f - \lambda  \in Q(g,h)}\}$.  
For each $d\in\N$, let:
\begin{equation}
\label{eq:rho-d}
  \rho_d:=\sup \{ {\lambda \in \R: \ f - \lambda  \in Q_d( g,h)} \}\,,\quad d\in\N\,,  
\end{equation}
where ${Q_d}(g,h)$ stands for the truncated quadratic module of order $d$:
\[Q_d(g,h): = \left\{ {\sum\limits_{j = 0}^m {{\sigma _j}{g_j}}  + \sum\limits_{t = 1}^l {{\phi _t}{h_t}} :\;{\sigma _j} \in \Sigma {[ x]_{d - {u_j}}}\,,\ {\phi_t} \in \R[x]_{2d - 2{w_t}}}\right\}\,,\]
with ${u_j}: = \lceil {\deg (g_j)/2} \rceil$  and ${w_t}: = \lceil {\deg( h_t)/2}\rceil $. For each fixed $d$, problem \eqref{eq:rho-d} is a reinforcement of \eqref{eq:def-dual}
and so $\rho_d\leq f^\star$ for all $d$. 
It also turns out that \eqref{def-pb} can also be written as:
\begin{equation}
   f^\star=\mathop {\inf }\limits_{\mu  \in \cM (S( g,h))} \int_{S(g,h)} {fd\mu } \,,
\end{equation}
where $\cM( {S(g,h)})$ is the set of all finite Borel measures supported on $S(g,h)$. Let us denote by $\R^{\N^n}$ the set of all real sequences ordered by ${\N^n}$. With $y=( y_\alpha)_{\alpha  \in {\N^n}}\in \R^{\N^n}$  being the moment sequence of a measure $\mu$, one has: 
\begin{equation}
    \label{eq:measure}
f^\star = \min \{ {  L_y(f) :\, y\in \R^{\N^n}{\text{ has a representing measure in }}\cM (S( g,h))} \}\,,
\end{equation}
and by Theorem \ref{theo:Putinar's.Positivstellensatz2}(ii):
\begin{equation}
    \label{eq:moment}
\begin{array}{rl}
{{f^\star} = \mathop {\min }\limits_{y \in {\R^\N }} }& L_y(f) \\
\qquad {{\text{s.t. }}}& M_d ( {{g_j}y}) \succeq 0\,, \, \forall d \in \N \,,\\
{}& M_d( {{h_t}y}) 
 = 0\,,\, \forall d \in \N\,,\\
 &y_0=1\,.
\end{array}\end{equation}
Let $d_{\min} := \max \{\lceil \deg(f) / 2  \rceil,  u_j, w_t \}$. For every fixed $d \geq d_{\min}$, consider the finite truncation of the above problem:
\begin{equation}
    \label{eq:mom-trunc}
    \begin{array}{rl}
{\tau_d} := \mathop {\min }\limits_{y \in {\R^{s( 2d)} }} & L_y(f)\\
\qquad \text{s.t. }& M_{d - u_j } ( {{g_j}y} ) \succeq 0\,,\\
&M_{d - w_t }( {{h_t}y} )   = 0\,,\\
 &y_0=1\,.
\end{array}
\end{equation}
Then \eqref{eq:mom-trunc} is a semidefinite relaxation of \eqref{eq:moment} and  is the dual of \eqref{eq:rho-d}. Moreover, strong duality holds according to Josz and Henrion \cite{josz2016strong}. 
Therefore, one has $\rho_d=\tau_d\leq f^\star$ for all $d$. 
This primal-dual sequence of semidefinite programs \eqref{eq:rho-d}-\eqref{eq:mom-trunc} 
is the so-called moment-SOS hierarchy for optimization (also known as ``Lasserre's hierarchy''), and
\[f^\star\,=\,\lim_{d\to\infty}\rho_d\,=\,\lim_{d\to\infty}\tau_d\,.\]
For more details on the moment-SOS hierarchy,  the interested reader is referred to \cite{lasserre2001global}. 
\if{
Geometrically speaking, one can approximate the closure of the convex hull of $S(g,h)$ by the set $\Omega_d$ of all $\bar x\in \R^n$ such that there exists a feasible solution $y$ of (\ref{eq:mom-trunc}) satisfying $\bar x_j=L_y(x_j)$, $j=1,\dots,n$. 
The set $\Omega_d$ is called a semidefinite representation of $S(g,h)$, which is firstly introduced in \cite{lasserre2009convex} and extended by Guo et al. \cite{guo2015semidefinite} for the case of non-compact $S(g,h)$ under the conditions of pointedness and closedness at infinity.
}\fi
\paragraph{Complexity of Putinar's Positivstellensatz.}
Let $c_n(\alpha) := \frac{|\alpha|!}{\alpha_1!\dots\alpha_n!}$ for each $\alpha \in \N^n$. 
We note $\| h \|_{\max}: = \mathop {\max }\limits_\alpha  \left\{ {{| {{h_\alpha }} |} / c_n(\alpha)} \right\}$, for a given $h \in \R[x]$.
The convergence rate of the sequence $(\rho_d)_{d\in\N}$ relies on the following result.
\begin{theorem}(Nie-Schweighofer,\cite{nie2007complexity})\label{theo:Complexity.Putinar}
Assume that $\emptyset \ne S(g,h)\subset (-1,1)^n$ is Archimedean and $f^\star>0$. Then there exists $C>0$ depending on $g$ and $h$ such that for $d\in\N$ and
\[d\ge C\exp \left( \left( \deg(f)^2 n^{\deg(f)}(f^\star )^{-1} \| f  \|_{\max} \right)^C \right)\,,\]
one has $ f \in  {Q_{ d}}( g,h )$.

\end{theorem}
\paragraph{Extraction of global minimizers.}
Assume that the optimal value $\tau_d$ is reached for a solution $y^\star$ and that this solution  satisfies the flat extension condition, that is: 
\[
\rank( {{M_d}( y^\star)}) = \rank( {M_{d - w}( y^\star)})\,,
\] 
for some $d\ge w:=\max \{ {{u_i},{w_j}}\}$.
Let $r:=\rank( {{M_d}( y^\star)})$ and let $\delta _a$ stands for the Dirac measure at point $a\in \R^n$. 
Then $\tau_d=f^\star$ and there exist $x^{(j)}\in\R^n$ and $\lambda_j\geq0$, $j=1,\dots,r$, with $\sum_{j = 1}^{r} {\lambda_j}=1$ such that and the sequence $y^\star$ has the representing $r$-atomic measure $\mu  = \sum_{j = 1}^{r} {{\lambda_j}{\delta _{  x^{(j)} }}} $. 
The support of $\mu$ is the set $\{ {x^{(j)}:\ j = 1,\dots,r} \}$, which belongs to the set of all minimizers of the original problem \eqref{def-pb}. 
Henrion and Lasserre  \cite{henrion2005detecting} provide a numerical algorithm to extract 
the $r$ atoms corresponding to the support of the atomic measure $\mu$.
\section{Representation theorems}
\label{sec:representations}
In this section we provide two exact representations of  globally nonnegative polynomials and polynomials nonnegative on basic semialgebraic sets (not necessarily compact). 
The representations are obtained thanks to a perturbation argument as well as existing representations for
positive definite forms. 
Let $\theta := 1 + {\| x \|_2^2}$. 
We denote by $\mathbb S^{n-1}$ the unit sphere in $\R^n$. For each $h\in\R[x]$, let
\[\delta(h):=\frac{{\sup \left\{ {h(x):x \in  \mathbb S^{n-1}} \right\}}}{{\inf \left\{ {h(x):x \in \mathbb S^{n-1}} \right\}}}\,.\]
%
For later use recall the following theorem.
\begin{lemma}(Reznick \cite[Theorem 3.12]{reznick1995uniform})\label{lem:homogeneous.sos}
Suppose that $p \in \R[ x]$ is a positive definite form of degree $2d$, for some $d\in\N$. 
Then for $k \in \N$ and
\[k\ge \frac{{2nd(2d - 1)}}{{4\log 2}}\delta (p) - \frac{{n + 2d}}{2}\,,\]
one has ${\| x \|^{2k}_2}p \in \Sigma {[ x ]_{k + d}}$.
\end{lemma}
\subsection{Globally nonnegative polynomials}
\label{sec:unconstrained}
Let us note $\|h\|_1:=\sum\nolimits_\alpha  {{|h_\alpha |}} $ for a given $h \in \R[x]$. The following result provides a representation of globally nonnegative polynomials. 
\begin{theorem}\label{theo:representation.nonnegative.poly}
Let $f\in \R[ x]_{2d}$ be nonnegative on $\R^n$. Then for every $\varepsilon>0$, for $k_\varepsilon \in \N$ and
\[k_\varepsilon\ge \frac{{2(n+1)d(2d - 1)}}{{4\log 2}}(\varepsilon^{-1}\|  f\|_{1}+1) - \frac{{n+1 + 2d}}{2}\,,\]
one has 
\begin{equation}\label{eq:certif.unconstrained}
{ \theta^{{k_\varepsilon }}}\left( {f + \varepsilon {{\theta}^{d}}} \right) \in \Sigma {[ x ]_{{k_\varepsilon } + d}}\,.
\end{equation}
\end{theorem}
The detailed proof of Theorem \ref{theo:representation.nonnegative.poly} postponed to Appendix~\ref{proof:representation.nonnegative.poly} consists of  three steps:
\begin{enumerate}
\item Associate a positive definite form to the globally nonnegative polynomial $f$.
\item Use Reznick's representation from Lemma \ref{lem:homogeneous.sos} to get a representation of this homogeneous form.
\item Transform back the homogeneous polynomial together with its representation to the original polynomial.
\end{enumerate}
\subsection{Polynomials nonnegative on a basic semialgebraic set}
\label{sec:constrained}
We recall the definition of the truncated quadratic module of order $d$ associated with $S(g)$:
\[Q_d(g): = \left\{ \sigma_0+\sum\limits_{j = 1}^m {{\sigma _j}{g_j}} :\;{\sigma _0} \in \Sigma {[ x]_{d}}\,,\, {\sigma _j} \in \Sigma {[ x]_{d - {u_i}}}\right\}\,,\]
where ${u_j}: = \lceil {\deg (g_j)/2} \rceil$, $j=1,\dots,m$.
For every $h \in \R[ x]$, let us define
\[d_1( h): =  {1 + \lfloor {{{\deg ( h )}}/{2}} \rfloor } \qquad \text{ and }\qquad d_2( h): =  {\lceil {{{\deg ( h )}}/{2}} \rceil } .\]
The following result provides a degree bound for the SOS multipliers of \cite[Theorem 1]{putinar1999positive}.
\begin{theorem}\label{theo:Representation.nonnegative2}
Let $g = \{ {{g_1},\dots,{g_m}} \} \subset \R[ x]$ and $f\in \R[ x]$ such that $f$ is nonnegative on $S( g )$. Let $\varepsilon  > 0$ and $d\in \N$ be such that  at least one of the following two conditions  is satisfied:

\noindent (i) $d\ge d_1(f)$;

\noindent (ii) $d\ge d_2(f)$ and $g_m:=f+\lambda$ for some real $\lambda \ge 0$.\\
Then there exist $k_\varepsilon\in \N$ such that
\begin{equation}\label{eq:bigger.quadratic2}
{\theta ^{{k_\varepsilon }}}\left( {f + \varepsilon\, {\theta ^{d}}}\right) \in  {Q_{{k_\varepsilon } + d}}( g )\,.
\end{equation}
\end{theorem}
The detailed proof of Theorem \ref{theo:Representation.nonnegative2} relies on Jacobi's technique in his proof of \cite[Theorem 7]{jacobi2001representation} and is postponed to Appendix~\ref{proof:Representation.nonnegative2}. This proof consists of three steps:
\begin{enumerate}
\item  Associate a homogeneous polynomial ${\tilde f}$ to the  polynomial $f$.
\item  Use Putinar's Positivstellensatz (Theorem \ref{theo:Putinar's.Positivstellensatz2} (i)) to obtain a representation of ${\tilde f}$.
\item  Transform back the representation of ${\tilde f}$ to obtain a representation of $f$.
\end{enumerate}
\begin{remark}
Theorem \ref{theo:Representation.nonnegative2} is an extension of  Putinar's Positivstellensatz to (possibly) non-compact sets $S(g)$, and so does not require the Archimedean condition. 
The price to pay for such an extension is the presence of the multiplier $\theta^{k_\varepsilon}$ in front of $f$ and the perturbation term $\varepsilon\,\theta^{d}$. Note that (ii) involves a tighter bound for $d$, compared to (i), since $d_1(f)\ge d_2(f)$. The counterpart is that (ii) requires to 
include the additional constraint $f+\lambda\ge 0$, for some $\lambda\ge 0$. 
\end{remark}
\paragraph{Complexity of Putinar-Vasilescu's Positivstellensatz.}

\if{
For every $h \in \R[ x ]$, we define by $\hat h$ the degree $2d_1 ( h )$ homogenization of $h$, i.e.,
\begin{equation}
\label{eq:homogenous.formula2}
\hat h( {x,{x_{n + 1}}} ) = x_{n + 1}^{2d_1 ( h )}h( {\displaystyle{x}/{{{x_{n + 1}}}}} )\,.
\end{equation}
where ${d_1}( h ): = {1 + \lfloor {{{\deg ( h )}}/{2}} \rfloor }$ 
and $\| \hat h \|_{\max,l}: = \mathop {\max }\limits_\alpha  \left\{ {{| {{h_\alpha }} |} / c_{n+1}(\alpha,2 d_1(h)-|\alpha|)} \right\}$, where $c_n(\alpha) = \frac{|\alpha|!}{\alpha_1!\dots\alpha_n!}$.
}\fi
For each $h\in\R[x]_{2l}$, we denote $\| h \|_{\max,l}: = \mathop {\max }\limits_\alpha  \left\{ {{| {{h_\alpha }} |} / c_{n+1}(\alpha,2l-|\alpha|)} \right\}$. Let us recall $c_n(\alpha) := \frac{|\alpha|!}{\alpha_1!\dots\alpha_n!}$ for each $\alpha \in \N^n$. 
Relying on Theorem \ref{theo:Complexity.Putinar}, one can analyze the complexity of Putinar-Vasilescu's Positivstellensatz as follows:
\begin{proposition}\label{prop:Complexity.Putinar-Vasilescu}
Assume that all assumptions of Theorem \ref{theo:Representation.nonnegative2} hold and $0_{\R^n}\in S(g)$. Then there exists $C>0$ depending on $g$ such that for all $k_\varepsilon\in \N$ satisfying
\[k_\varepsilon\ge C\exp \left( \left( 4^{d+1}d^2(n + 1)^{2d}\left(\varepsilon^{-1} \| f  \|_{\max, d}+ \max\limits_{\bar \alpha \in \N^{n+1}_d} \displaystyle \frac{c_{n+1}(\bar \alpha)}{c_{n+1}(2\bar \alpha)} \right) \right)^C \right)-d\,,\]
one has ${\theta ^{{k_\varepsilon }}}\left( {f + \varepsilon\, {\theta ^{d}}}\right) \in  {Q_{{k_\varepsilon } + d}}( g )$.
\end{proposition}
The proof of Proposition~\ref{prop:Complexity.Putinar-Vasilescu} is postponed to Appendix~\ref{proof:Complexity.Putinar-Vasilescu}.
\paragraph{Discussion about the $\varepsilon$ parameter.}
The (arbitrary small) positive parameter $\varepsilon$ in Theorem \ref{theo:representation.nonnegative.poly} and Theorem \ref{theo:Representation.nonnegative2} ensures the positivity of polynomials over the respective considered domain $\R^n$ or $S(g)$,   excluding the origin in the homogenized  representations. However these representations can still  hold, even when $\varepsilon=0$, as illustrated in the following two examples: 
\begin{example}\label{exam:zero.parameter.epsilon}
\noindent (i) Motzkin's polynomial $f = x_1^4x_2^2 + x_1^2x_2^4 + 1 - 3x_1^2x_2^2$ is globally nonnegative but not SOS. However, $\theta f$ is SOS since
\[\begin{array}{rl}
\theta f =& 2{( {\frac{1}{2}x_1^3{x_2} + \frac{1}{2}{x_1}x_2^3 - {x_1}{x_2}} )^2} + {( {x_1^2{x_2} - {x_2}} )^2} + {( {{x_1}x_2^2 - {x_1}} )^2}\\
 &+ \frac{1}{2}{( {x_1^3{x_2} - {x_1}{x_2}} )^2} + \frac{1}{2}{( {{x_1}x_2^3 - {x_1}{x_2}} )^2} + {( {x_1^2x_2^2 - 1} )^2}\,.
\end{array}\]
\noindent (ii) Let $f = \left( {x_1^2 + x_2^2} \right)x_1^2x_2^2 - 3x_1^2x_2^2$ and  $g = x_1^2 + x_2^2 - 4$. 
It is not hard to show that $f$ is nonnegative on the non-compact set $S( g)$. 
Moreover, $f = \frac{1}{4}x_1^2x_2^2\left( {x_1^2 + x_2^2} \right) + \frac{3}{4}x_1^2x_2^2g$. Thus, ${\theta ^0}f  \in {Q_3}\left( g \right)$.
\end{example}
However,  the certificate \eqref{eq:certif.unconstrained} for global nonnegativity with $\varepsilon=0$ is not true in general, as shown in the following lemma:
\begin{lemma}\label{lem:positi.eps.zero}
The nonnegative dehomogenized Delzell's polynomial \[f=x_1^4x_2^2+x_2^4x_3^2+x_1^2x_3^4-3x_1^2x_2^2x_3^2+x_3^8\]
satisfies that ${\theta ^k}f \notin \Sigma[x]$ for all $k\in \N$.
\end{lemma}
\begin{proof}
Assume by contradiction that ${\theta ^K}f \in \Sigma[x]$ for some $K\in \N$. Note that $n=3$ here. We denote by $\tilde f$ the degree $8$ homogenization of $f$, i.e., 
\[\tilde f=x_4^2(x_1^4x_2^2+x_2^4x_3^2+x_1^2x_3^4-3x_1^2x_2^2x_3^2)+x_3^8\,.\]
Then ${\|(x,x_{n+1})\|_2^{2K}}\tilde f \in \Sigma[x,x_{n+1}]$. As shown in \cite[\S 6]{reznick2000some}, it is impossible. This contradiction yields the conclusion.
\end{proof}
The certificate \eqref{eq:bigger.quadratic2} for global nonnegativity on basic semialgebraic sets  with $\varepsilon=0$ is also not true in general, as shown in the following lemma:
\begin{lemma}\label{lemma:not.first.order}
With $n=1$, let $f=x$ and $g = \{ {{x^3}, - {x^3}} \}$. Then $f = 0$ on $S( g ) = \{ 0 \}$. It follows that $f$ is nonnegative on $S( g )$, but:

\noindent (i) ${\theta ^k}f \notin Q( g )$ for all $k\in \N$ and;

\noindent (ii) for every $\varepsilon>0$, ${\theta ^k}(f + \varepsilon \theta ) \in Q_{k+1}( g )$ for all $k\in\N$ with $k \ge \max\{2\,,\,{\varepsilon ^{ - 2}}/4 - 1\}$.
\end{lemma}
\begin{proof}
We will show statement (i). 
Assume by contradiction that there exists $k\in\N$ such that ${\theta ^k}f \in Q( g )$. 
Then there exists ${q_j}( x ) \in \R[ x ]$, $j=0,\dots,r$ such that 
\[{\theta ^k}f = \sum\limits_{j = 1}^m {{q_j}{{( x )}^2}}  + {q_0}( x ){x^3}.\]
Assume that ${q_j}( x ) = {a_j} + {b_j}x + {x^2}{d_j}( x )$, where $a_j,b_j\in\R$ and $d_j\in \R[x]$, $j=1,\dots,r$. 
From this and since ${\theta ^k} = 1 + {x^2}e(x)$ for some $e\in\R[x]$, one has
\[(1 + {x^2}e(x))x = \sum\limits_{j = 1}^r {a_j^2}  + 2\sum\limits_{j = 1}^r {{a_j}{b_j}x}  + {x^2}p(x)\,,\]
for some $p\in\R[x]$. 
By comparing coefficients of monomials $1$ and $x$ in the two sides of the above equality,
$\sum\limits_{j = 1}^r {a_j^2}  = 0$ and $2\sum\limits_{j = 1}^r {{a_j}{b_j}}  = 1$. 
It implies that $a_j=0$, $j=1,\dots,r$, and $2\sum\limits_{j = 1}^r {{a_j}{b_j}}  = 1$. 
It follows that $0=1$. 
It is impossible. 

Let us prove the statement (ii). Let $\varepsilon>0$ and $k\in \N$, $k\ge 2$. Since ${\theta ^k} = 1 + kx^2 + {x^4}e(x)$ for some $e\in\R[x]_{2k-4}$, one has
\[{\theta ^k}(f + \varepsilon \theta ) = (1 + k{x^2} + {x^4}e(x))(\varepsilon  + x + \varepsilon {x^2})= \varepsilon  + x + \varepsilon ( {k + 1} ){x^2} + {x^3}q( x )\,,\]
 for some $q\in\R[x]_{2k-2}$. 
Assume that $k \ge {\varepsilon ^{ - 2}}/4 - 1$. 
Then
\[{\theta ^k}(f + \varepsilon \theta ) = \varepsilon  - \frac{1}{{4\varepsilon (k + 1)}} + {\left( {x\sqrt {\varepsilon (k + 1)}  + \frac{1}{{2\sqrt {\varepsilon (k + 1)} }}} \right)^2} + {x^3}q( x )\in Q_{k+1}( g )\,.\]
\end{proof}
%

From Lemma \ref{lem:positi.eps.zero} and Lemma \ref{lemma:not.first.order}, we conclude that the strict positivity of the $\varepsilon$ parameter is necessary in general although the certification with $\varepsilon=0$ may happen in many cases.

When the certificate \eqref{eq:certif.unconstrained} with $\varepsilon=0$ occurs, one has the following remark about the exponent of $\theta$ in \eqref{eq:certif.unconstrained}.
\begin{remark}
If $n=2$, there does not exist a fixed $K\in \N$ such that for all  nonnegative $f\in\R[x]_6$ , $\theta^K f$ is SOS. 
Indeed, assume by contradiction that there exists such a $K$. Then, the degree $6$ homogenization $\tilde f$ of $f$ would be a positive ternary sextic such that $\|(x,x_{n+1})\|^{2K}_2 \tilde f$ is SOS. 
By using \cite[Theorem 1]{reznick2005absence} and the fact that the homogeneous Mozkin's polynomial is a  positive ternary sextic which is not SOS, we obtain a contradiction. 
\end{remark}
In certificate \eqref{eq:certif.unconstrained} with $\varepsilon=0$, the multiplier $\theta^{k_{\varepsilon}}$ can be replaced with other kinds of SOS in some certain examples, e.g., Delzell's, Leep-Starr's  from \cite{kaltofen2012exact, schabert2019uniform}.
\section{Polynomial optimization}
\label{sec:pop}
In this section, we exploit the two representations from Theorem \ref{theo:representation.nonnegative.poly} and Theorem \ref{theo:Representation.nonnegative2} to construct new hierarchies of semidefinite programs for POPs of the form ${f^\star} = \inf\{ {f( x ):\ x \in \Omega} \}$ where $\Omega=\R^n$ for the unconstrained case and $\Omega=S\left(g\right)$ for the constrained case (with no compactness assumption), respectively.  
Instead of solving the original problem, we are rather interested in the perturbed problem:
\begin{equation}
    \label{eq:perturbed}
f_\varepsilon ^\star := \inf\{ {f( x ) + \varepsilon\,\theta {{( x )}^d}:\ x \in \Omega} \}\,,
\end{equation}
where $\varepsilon>0$ is fixed, $\theta(x) := 1 + {\| x \|_2^2}$, and $2d \ge \deg \left( f \right)$. 
Now, assume that the optimal value $f^\star$ of the original problem is attained at some $x^\star\in \Omega$.
It is not difficult to show that if $\Omega$ is unbounded, the polynomial $f + \varepsilon \,{\theta ^d}$ is coercive on $\Omega$, i.e.,
\[\mathop {\lim }\limits_{x\in \Omega\,,\,\| x \|_2 \to \infty } \left( {f( x ) + \varepsilon \theta {{( x )}^d}} \right) = \infty \,,\]
(see more in \cite{bajbar2015coercive}). 
Indeed, it is due to the fact that $f$ is bounded from below by $f^\star$ on $\Omega$ and $\theta {{( x )}^d}\to \infty$ as $\| x \|_2 \to \infty$. 
Thus, the optimal value $f_\varepsilon ^\star$ of the perturbed problem \eqref{eq:perturbed} is always attained at some global minimizer $x^\star_\varepsilon$ even if $\Omega$ is non-compact.  
Then:
\[\begin{array}{rl}
{f^\star} + \varepsilon \,\theta {( {{x^\star}} )^d} &= f( {{x^\star}} ) + \varepsilon \,\theta {( {{x^\star}})^d}\\
 &\ge f_\varepsilon ^\star = f( {x_\varepsilon ^\star} ) + \varepsilon \,\theta {( {x_\varepsilon ^\star} )^d} \ge f( {x_\varepsilon ^\star} ) \ge {f^\star}.
\end{array}\]
Thus, ${f^\star} \in [ {f_\varepsilon ^\star - \varepsilon\, \theta {{( {{x^\star}} )}^d},f_\varepsilon ^\star} ]$, i.e., $f^\star_\varepsilon$ is a perturbation of $f^\star$ and the gap between both of them is at most $\varepsilon\, \theta ( x^\star)^d$. 
Next, observe that:
\[\begin{array}{rl}
f_\varepsilon ^\star &= \sup \{ {\lambda  \in \R:\,f + \varepsilon\, {\theta ^d} - \lambda  \ge 0{\text{ on }}\Omega} \}\\
& = \sup \{ {\lambda  \in \R:\,{\theta ^k}( {f + \varepsilon\, {\theta ^d} - \lambda } ) \ge 0{\text{ on }}\Omega} \},\ k\in \N\,.
\end{array}\]
The following hierarchies are based on the simple idea of replacing constraint ``$\theta ^k( f + \varepsilon {\theta ^d} - \lambda  ) \ge 0{\text{ on }}\Omega$" by relaxed constraint``${\theta ^k}( {f + \varepsilon {\theta ^d} - \lambda } )$ is in the truncated quadratic module associated with  $\Omega$".
\subsection{Unconstrained case}
\label{sec:unconstrained_hierarchy}
Given $f \in \R[ x]_{2 d}$, let us consider the following problem:
\begin{equation}\label{eq:unc.constrained.problem}
\begin{array}{l}
f^\star:=\inf\limits_{x \in {\R^n}} f( x )\,.
\end{array}
\end{equation}
In the sequel, we assume that $f^\star>-\infty$ and let $\varepsilon>0$ be fixed.  
\begin{theorem}\label{theo:uncconstr.theo}
Consider the hierarchy of semidefinite programs indexed by $k\in\N$:
\begin{equation}\label{eq:unc.primal.problem}
{\rho_k^1}(\varepsilon) := \sup \{ {\lambda  \in \R:\ {\theta ^k}( {f- \lambda +  \varepsilon {\theta ^d}  } ) \in \Sigma {[ x ]_{k + d}}} \}\,.
\end{equation}
The following statements hold:
\begin{enumerate}
\item The sequence ${( {{\rho_k^1}}(\varepsilon))_{k \in \N}}$ is monotone non-decreasing.
\item Assume that $f^\star$ in \eqref{eq:unc.constrained.problem} is attained at $x^\star\in\R^n$. Then there exists $K\in \N$ such that $f^\star\le {\rho_k^1} (\varepsilon)\le {f^\star} +  \varepsilon\, \theta {( {{x^\star}})^d}$ for all $k\ge K$. In particular, $K$ is upper bounded by $ O(\varepsilon^{-1})$ as $\varepsilon \downarrow 0$.
\end{enumerate}
\end{theorem}
\begin{proof}
\begin{enumerate}
\item Let $k\in \N$ and fix $\bar \varepsilon>0$, arbitrary. 
By \eqref{eq:unc.primal.problem}, there exists a real $\bar \lambda$ such that
\[{\rho _k^1}( \varepsilon) - \bar \varepsilon  \le \bar \lambda\text{ and }{\theta ^k}( {f-  \bar\lambda  + \varepsilon {\theta ^d} } ) \in \Sigma {[ x ]_{k + d}}\,.\]
Since $\theta \in \Sigma {[ x ]_{1}}$, ${\theta ^{k + 1}}( {f -  \bar \lambda +  \varepsilon {\theta ^d} } ) \in \Sigma {[ x ]_{k + d+1}}$. By \eqref{eq:unc.primal.problem}, ${\rho ^1_{k + 1}}( \varepsilon) \ge  \bar \lambda  \ge {\rho^1 _k}( \varepsilon) - \bar \varepsilon $. 
This implies that ${\rho ^1_{k + 1}}( \varepsilon) \ge {\rho^1 _k}( \varepsilon)$.
\item By \eqref{eq:unc.constrained.problem}, $f-f^\star$ is nonnegative. 
By Theorem \ref{theo:representation.nonnegative.poly}, there exists $K\in \N$ and $K= O(\varepsilon^{-1})$ as $\varepsilon \downarrow 0$ such that for all $k\ge K$,
\[{\theta ^K}( {f  - f^\star +  \varepsilon {\theta ^d}} ) \in \Sigma [ x ]_{K + d}\,.\]
By \eqref{eq:unc.primal.problem}, ${f^ \star } \le {\rho^1 _k}( \varepsilon)$ for all $k\ge K$.
Let $k\in \N$ and fix $\bar \varepsilon>0$, arbitrary. 
By \eqref{eq:unc.primal.problem}, there exists a real $\bar \lambda$ such that
\[{\rho ^1_k}( \varepsilon) - \bar \varepsilon  \le  \bar \lambda\text{ and }{\theta ^k}( {f  -  \bar \lambda +  \varepsilon {\theta ^d}} ) \in \Sigma [ x ]_{k + d}\,.\]
It follows that $f -  \bar \lambda +  \varepsilon {\theta ^d}  \ge 0$ on $\R^n$. 
From this, 
\[f^\star +  \varepsilon \theta {( {{x^\star}} )^d}=f( {{x^\star}} ) +  \varepsilon \theta {( {{x^\star}} )^d} \ge \bar  \lambda  \ge {\rho^1 _k} -\bar \varepsilon \,.\]
This implies $f^\star +  \varepsilon \theta {( {{x^\star}} )^d} \ge {\rho ^1_k}( \varepsilon)$, the desired result.
\end{enumerate}
\end{proof}
For every $k\in\N$, the dual of \eqref{eq:unc.primal.problem} reads:
\begin{equation}\label{eq:unc.dual}
\begin{array}{rl}
{\tau_k^1}(\varepsilon): = \inf &{L_y}( {{\theta ^k}( {f + \varepsilon {\theta ^d}} )} )\\
\text{s.t.}&y = {( {{y_\alpha }} )_{\alpha  \in \N^n_{2( {d + k} )}}} \subset \R\,,\\
&{M_{k + d}}( y )\succeq 0\,,\\
&{L_y}( {{\theta ^k}} ) = 1\,.
\end{array}
\end{equation}
We guarantee strong duality for previous primal-dual problems:
 \begin{proposition}\label{prop:strong.duality.unc}
Let  $k\in \N$. Then $ \rho_k^1(\varepsilon)=\tau_k^1(\varepsilon)$. 
Moreover, if  $\tau_k^1(\varepsilon)>-\infty$ then the 
optimal value $\rho_k^1(\varepsilon)$ is attained.
\end{proposition}
\begin{proof}
By Slater's constraint qualification  \cite[\S 5.2.3]{boyd2004convex}, it suffices to show that \eqref{eq:unc.dual} admits a strictly feasible solution.
 Let us denote by $\mu$ the measure with  density ${\chi _{{{[ {0,1} ]}^n}}}\theta^{-k}$ with respect to the Lebesgue measure, where $\chi _A$ is the characteristic function of a given set $A\subset \R^n$. 
Set ${y_\alpha }: = \int {{x^\alpha }d\mu }  $ for all $\alpha  \in \N^n$. 
We claim that $y_\alpha\in\R$ for all $\alpha  \in \N^n$, ${L_y}( {{\theta ^k}} ) = 1$ and ${M_{k + d}}( y )\succ 0$.
 Indeed, for all $\alpha  \in \N^n$
\[\begin{array}{rl}
| {{y_\alpha }} | &= | {\int {{x^\alpha }{\chi _{{{[ {0,1}]}^n}}}{\theta ^{ - k}}dx} } | = | {\int_{{{[ {0,1} ]}^n}} {{x^\alpha }{\theta ^{ - k}}dx} } |\\
 &\le \int_{{{[ {0,1} ]}^n}} {{{| {{x_1}} |}^{{\alpha _1}}}\dots{{| {{x_n}} |}^{{\alpha _n}}}{\theta ^{ - k}}dx}  \le 1\,,
\end{array}\]
since $\theta^{-k}\le 1$. 
Thus, $y_\alpha\in\R$ for all $\alpha  \in \N^n$. 
In addition,
\[{L_y}( {{\theta ^k}}) = \int {{\theta ^k}{\chi _{{{[ {0,1}]}^n}}}{\theta ^{ - k}}dx}  = \int_{{{[ {0,1}]}^n}} {dx}  = 1\,.\]
Let $p\in {\R^{s(d+k)}}\backslash \{ 0 \}$ be fixed.
 We state that ${p^T}{M_{d + k}}( y )p>0$. 
Assume by contradiction that ${p^T}{M_{d + k}}( y)p\le0$.
 One has
\[\begin{array}{rl}
0 &\ge {p^T}{M_{d + k}}( y )p = \int {{p^T}{v_{d + k}}v_{d + k}^Tpd\mu } \\
 &= \int {{p^T}{v_{d + k}}v_{d + k}^Tp{\chi _{{{[ {0,1}]}^n}}}{\theta ^{ - k}}dx}  = \int_{{{[ {0,1}]}^n}} {{{( {{p^T}{v_{d + k}}})}^2}{\theta ^{ - k}}dx} \ge 0\,.
\end{array}\]
It follows that ${p^T}{v_{d + k}}=0$ on ${{[ {0,1}]}^n}$, thus $p=0$  yielding a contradiction. 
From this, ${( {{y_\alpha }} )_{\alpha  \in \N_{d + k}^n}}$ is a feasible solution of \eqref{eq:unc.dual} with ${M_{k + d}}( y )\succ 0$. 
By strong duality, the conclusion follows.
\end{proof}

\subsection{Constrained case}
\label{sec:constrained_hierarchy}
Consider the following problem:
\begin{equation}\label{eq:constrained.problem}
\begin{array}{l}
f^\star:=\inf\limits_{x \in S( g)} f( x)\,,
\end{array}
\end{equation}
where $f \in \R[ x ]$, $g = \{ {{g_1},\dots,{g_m}} \} \subset \R[ x]$. Assume that $S( g) \ne \emptyset $ and $f^\star>-\infty$. Denote ${u_j}: = \lceil {\deg ( {{g_j}} )/2} \rceil $, $j=0,1,\dots,m$. Let $\varepsilon>0$ be fixed. 
\subsubsection{Unknown lower bound}
Let $d: = \lfloor \deg(f)/2\rfloor +1$.
\begin{theorem}\label{theo:constr.theo}
Consider the hierarchy of semidefinite programs indexed by $k\in\N$: 
\begin{equation}\label{eq:primal.problem}
{\rho _k^2} (\varepsilon):= \sup \{ {\lambda  \in \R:\ {\theta ^k}\,( {f - \lambda +  \varepsilon \,{\theta ^d} } ) \in Q_{k + d}( g)} \}\,.
\end{equation}
The following statements hold:
\begin{enumerate}
\item The sequence ${( {{\rho_k^2}(\varepsilon)})_{k \in \N}}$ is monotone non-decreasing.
\item Assume that problem \eqref{eq:constrained.problem} has an optimal solution $x^\star$. 
Then there exists $K\in \N$ such that $f^\star\le {\rho_k^2}(\varepsilon) \le {f^\star} +  \varepsilon\, \theta {( {{x^\star}})^d}$ for all $k\ge K$. In particular, if $0_{\R^n}\in S(g)$, then $K$ is upper bounded by $C \exp \left( O(\varepsilon^{-1})^C \right) -d$ as $\varepsilon  \downarrow  0$, for some $C>0$ depending on $g$.
\end{enumerate}
\end{theorem}
The proof of Theorem~\ref{theo:constr.theo} relies on Theorem \ref{theo:Representation.nonnegative2} (i) and is similar to Theorem~\ref{theo:uncconstr.theo}. 
The upper bound on $K$ is based on Proposition \ref{prop:Complexity.Putinar-Vasilescu}.
 
 For every $k\in\N$, the dual of \eqref{eq:primal.problem} reads:
\begin{equation}\label{eq:dual-sdp}
\begin{array}{rl}
{\tau_k^2}(\varepsilon): = \inf &{L_y}( {{\theta ^k}( {f +  \varepsilon {\theta ^d}} )} )\\
\text{s.t.}&y = {(y_\alpha )_{\alpha  \in \N^n_{2( {d + k})}}} \subset \R\,,\\
&{M_{k + d}}( {y})\succeq 0\,,\\
&{M_{k + d - {u_j}}}( {{g_j}y}) \succeq 0 ,\;j = 1,\dots,m\,,\\
&{L_y}( {{\theta ^k}}) = 1\,.
\end{array}
\end{equation}
We guarantee strong duality for previous primal-dual problems:
\begin{proposition}\label{prop:strong.duality.con}
There exists $K\in \N$ such that $\rho_k^2(\varepsilon)=\tau_k^2(\varepsilon)$ for all $k\ge K$. 
Moreover, if  $\tau_k^2(\varepsilon)>-\infty$, the optimal value $\rho_k^2(\varepsilon)$ is attained.
\end{proposition}
The proof of Proposition~\ref{prop:strong.duality.con} is postponed to Appendix~\ref{proof:strong.duality.con}.

\begin{remark}
If $S(g)$ has nonempty interior then strong duality holds for all orders $k$ of the  primal-dual problems \eqref{eq:primal.problem}-\eqref{eq:dual-sdp}. Indeed, by constructing a sequence of moments from the Lebesgue measure on an open ball contained in $S(g)$, one can find a strictly feasible solution of \eqref{eq:dual-sdp} and then apply Slater's constraint qualification  \cite[\S~5.2.3]{boyd2004convex}.
\end{remark}
\subsubsection{Known lower bound}
Assume that $g_m:=f-\underline f$ for some real $\underline f\le f^\star$ and let $d: = \lceil\deg( f)/2\rceil$. 
We then obtain the same conclusion as Theorem~\ref{theo:constr.theo} with replacing here $\tau_k^2(\varepsilon)$ and $\rho_k^2(\varepsilon)$ by $\tau_k^3(\varepsilon)$ and  $\rho_k^3(\varepsilon)$, respectively.
The proof relies on Theorem \ref{theo:Representation.nonnegative2} (ii) and is similar to Theorem~\ref{theo:uncconstr.theo}. 
Note that here $g_m=(f-f^\star)+(f^\star-\underline f )$ with $f-f^\star\ge 0$ on $S(g)$ and $f^\star-\underline f\ge 0$. 
The upper bound on $K$ is also based on Proposition \ref{prop:Complexity.Putinar-Vasilescu}.

The next proposition states that strong duality is guaranteed for each relaxation order $k$.
\begin{proposition}\label{prop:strong.duality.con2}
Let  $k\in \N$. 
Then  $\rho_k^3(\varepsilon)=\tau_k^3(\varepsilon)$. 
Moreover, if  $\tau_k^3(\varepsilon)>-\infty$ then the optimal value $\rho_k^3(\varepsilon)$ is attained.
\end{proposition}
The proof of Proposition~\ref{prop:strong.duality.con2} is postponed to Appendix~\ref{proof:strong.duality.con2}.
\begin{remark}
A lower bound $\underline f$ of problem \eqref{eq:constrained.problem} can be obtained by solving the following SDP: 
\[\sup \{ {\lambda  \in \R:\ {\theta ^k}\,( {f - \lambda  } ) \in Q_{k + d}( g)} \}\,,\ k \in \N\,.\]
Assume that we know a lower bound $\underline f$ of problem \eqref{eq:constrained.problem}. By adding the inequality constraint $f-\underline f\geq0$ in $S(g)$, we obtain the same semialgebraic set, i.e., $S( {g \cup \{ {f - \underline f } \}} ) = S(g)$. Thus, the two problems $\inf\{f(x):\,x\in S(g)\}$ and $\inf\{f(x):\,x\in S( {g \cup \{ {f - \underline f } \}} )\}$ are identical and the primal-dual SDP relaxations with values $\rho_k^3(\varepsilon)$ and $\tau_k^3(\varepsilon)$ of the latter one satisfy strong duality for each relaxation order $k$.
\end{remark}
\paragraph{General case.} 
Since ${g_j} \ge 0$ on $S(g)$ and  $- {g_j} \ge 0$  on $S(g)$ is equivalent to ${g_j} = 0$ on $ S(g)$, $ S(g)$ can be rewritten as $ S(g,h)$ with $g =\{ {{g_1},\dots,{g_m}} \}$ is the set of polynomials involved in the inequality constraints and $h = \{ {{h_1},\dots,{h_l}} \} $ is the set of polynomials involved in the equality constraints; in addition, $\R^n=S( {\{ 0 \},\emptyset } )$. Consider the general POP:
\begin{equation}\label{eq:gen.POP}
f^\star:=\inf\limits_{x \in S( g,h)} f( x)\,,
\end{equation}
with $f^\star\in\R$ and define 
\[(d,i) = \left\{ \begin{array}{rl}
(\lceil {\deg (f)/2} \rceil,1) &\text{ if }S(g,h) = {\R^n}\,,\\
(1 + \lfloor {\deg (f)/2} \rfloor,2)&\text{ if }S(g,h) \ne {\R^n}\text{ and lower bound $\underline f$ is unknown}\,,\\
(\lceil {\deg (f)/2} \rceil,3) &\text{ otherwise and set }g_m:=f-\underline f\,.
\end{array} \right.\]
For fixed $\varepsilon>0$, one considers the following SDP relaxation of POP \eqref{eq:gen.POP}:
\begin{equation}\label{eq:primal2}
{\rho_k^i}(\varepsilon) := \sup \{ {\lambda  \in \R:\ {\theta ^k}( {f - \lambda +  \varepsilon {\theta ^d} } ) \in Q_{k + d}( g,h)} \}\,.
\end{equation}
where $g_0:=1$, ${u_j} = \lceil {\deg ( {{g_j}} )/2}\rceil $, ${w_t} = \lceil {\deg ( {{h_t}} )/2} \rceil$. 
Then $\rho_k^i( \varepsilon)$ is an upper bound of $f^\star$ when $k$ is sufficiently large. 
The dual of \eqref{eq:primal2} is the semidefinite program:
\begin{equation}\label{eq:dual2}
\begin{array}{rl}
{\tau_k^i}(\varepsilon): = \inf &{L_y}( {{\theta ^k}( {f +  \varepsilon {\theta ^d}} )} )\\
\text{s.t.}&y = {( {{y_\alpha }} )_{\alpha  \in \N^n_{2( {d + k} )}}} \subset \R\,,\\
&{M_{k + d - {u_j}}}( {{g_j}y} ) \succeq 0 \,,\;j = 0,\dots,m\,,\\
&{M_{k + d - {w_t}}}\left( {{h_t}y} \right) = 0 \,,\;t = 1,\dots,l\,,\\
&{L_y}( {{\theta ^k}} ) = 1\,,
\end{array}
\end{equation}
The zero-duality gap between SDP~\eqref{eq:dual2} and SDP~\eqref{eq:primal2} is guaranteed for large enough $k$.
\begin{remark}
The condition $\tau_k^i(\varepsilon)>-\infty$ is always satisfied whenever $k$ is sufficently large. 
Indeed by weak duality,  when $\varepsilon$ is fixed and $k$ is sufficiently large then $\tau_k^i(\varepsilon) \ge \rho_k^i(\varepsilon) \ge f^\star >-\infty$.
However, when $k$ is small, $\tau_k^i(\varepsilon) = -\infty$ may happen.
\end{remark} 
Let us now assume that the POP \eqref{eq:gen.POP} has an optimal  solution $x^\star$. 
Then the gap between  $\rho_k^i( \varepsilon)$, and $f^\star$ is at most $\varepsilon\, \theta {( {{x^\star}})^d}$. 
Therefore, $\rho _k^i ( \varepsilon)$ is indeed an approximation of $ {f^\star}$. In practice, $(\rho_k^i( \varepsilon))_{k\in\N}$ often converges to the optimal value $f_\varepsilon ^\star := \min\{ {f( x ) + \varepsilon\,\theta {{( x )}^d}\,:\, x \in S(g,h)} \}$ after finitely many steps (see Section~\ref{sec:benchs}).
\begin{remark}
\noindent (i) The term $\theta^d$ in both \eqref{eq:primal2}  and \eqref{eq:dual2} can be replaced by $\varphi_d(x,1)$ where $\varphi_d :\R^{n+1}\to \R$ is a positive form of degree $2d$. For instant,  one can select $\varphi_d(x,1)=x_1^{2d}+\dots+x_n^{2d}+1$.  

\noindent (ii) Let $r\in \N$ be fixed. For every $k$ divisible by $2r$, the term $\theta^k$ appearing in both \eqref{eq:primal2} and \eqref{eq:dual2} can be replaced by $\psi_r(x,1)^{k/(2r)}$ where $\psi_r :\R^{n+1}\to \R$ is a coercive positive form of degree $2r$. For instant,  one can select $\psi_r(x,1)=x_1^{2r}+\dots+x_n^{2r}+1$.
\end{remark}
\paragraph{\textbf{Relation between  classical optimality conditions and nonnegativity certificates}.} 
When $\rho _k^i (0)= f^\star$, the constraint qualification conditions hold at $x^\star$.
\begin{proposition}\label{prop:optimal.cond.positi.certificate}
Assume that $\rho _k^i (0)= f^\star$ for some $k\in\N$, i.e., there exists $\sigma_j\in \Sigma[x]$, $j=0,\dots,m$ and $\phi_t\in\R[x]$, $t=1,\dots,l$ such that ${\theta ^k}(f - {f^\star}) = {\sigma _0} + \sum\limits_{j = 1}^m {{\sigma _j}{g_j}}  + \sum\limits_{t = 1}^l {{\phi _t}{h_t}}$. Then the constraint qualification conditions hold at $x^\star$: 
\begin{enumerate}
\item $\sigma_j(x^\star)\ge 0$ and $g_j(x^\star)\ge 0$, for all  $j=1\dots,m$;
\item $\sigma_j(x^\star)g_j(x^\star)=0$, for all $j=1,\dots,m$;
\item ${\theta(x^\star) ^k}\nabla  f (x^\star) = \sum\limits_{j = 1}^m {{\sigma _j(x^\star)}\nabla{g_j}(x^\star)}  + \sum\limits_{t = 1}^l {{\phi _t(x^\star)}\nabla{h_t(x^\star)}}$.
\end{enumerate}
\end{proposition}
The proof of Proposition \ref{prop:optimal.cond.positi.certificate} is similar to \cite[Theorem 7.4]{lasserre2015introduction}.

If we take an arbitrary small $\varepsilon>0$ then $\rho _k^i ( \varepsilon)$ is arbitrary close to $f^\star$ for large enough $k$. However, if one sets $\varepsilon=0$, the statement ``$\rho _K^i (0)=f^\star$ for some $K\in\N$" is not true in general as stated in the following proposition: 
\begin{proposition}\label{prop:finite.convergence}
If the first order optimality condition fails at a global minimizer of problem \eqref{eq:gen.POP}, then $\rho _k^i (0)< f^\star$ for all $k\in\N$.
\end{proposition}
The proof of Proposition \ref{prop:finite.convergence} is similar to \cite[Proposition 3.4]{nie2014optimality}.

By applying Proposition \ref{prop:finite.convergence} to POP $\min \{x\,:\, x^3=0\}$, we obtain the statement (i) of Lemma \ref{lemma:not.first.order}. 
Indeed, the first order optimality condition fails at the global minimizer $0$ of  this problem. 
Therefore, the positivity of $\varepsilon$ 
ensures convergence of $(\rho _k^i (\varepsilon))_{k\in\N}$ to the neighborhood $[f^\star,f^\star+\varepsilon \theta (x^\star)^d]$ of the optimal value $f^\star$.

We also conjecture that $\rho _K^i (0)=f^\star$ for some $K\in\N$ when some classical optimality conditions hold at every global minimizer of \eqref{eq:gen.POP}. 
In many cases, $\rho_K^i(0)=f^\star$ with $K=0,1$ when the KKT conditions hold (see Example~\ref{exam:zero.parameter.epsilon} and \cite[Example 4.4]{nie2012discriminants} with  $x_n=1$). 
However, the KKT conditions are not enough for this conjecture due to the fact that the minimizer $x^\star=(0,0,0)$ of  dehomogenized  Delzell's polynomial in Lemma \ref{lem:positi.eps.zero} satisfies the KKT conditions and $\rho _k^i (0)<f^\star$ for all $k\in\N$ in this case.
\paragraph{\textbf{Reducing the non-compact case to a compact case}.} 
Consider the POP: $f^\star:=\inf \{f(x)\,:\,x\in S(g,h)\}$ where the feasible set $S(g,h)$ is possibly non-compact, and the associated perturbed POP: $f^\star_\varepsilon:=\inf \{f(x)+\varepsilon \theta(x)^d \,:\,x\in S(g,h)\}$ with fixed $\varepsilon>0$. Here one assumes that $f^\star$ is attained at $x^\star$ and $2d>\deg(f)$. As in Section \ref{sec:pop}, ${f^\star} \in [ {f_\varepsilon ^\star - \varepsilon\, \theta {{( {{x^\star}} )}^d},f_\varepsilon ^\star} ]$. Suppose that a point $\bar x$ in $S(g,h)$ is known. It is not hard to show that $f+\varepsilon \theta^d$ is coercive and therefore with $C:=f(\bar x)+\varepsilon \theta(\bar x)^d$, the set  $S(\{C-f-\varepsilon \theta^d\})$ is compact. Moreover,
\begin{equation} \label{eq:compact.POP.form}
f^\star_\varepsilon=\inf \{f(x)+\varepsilon \theta(x)^d \,:\,x\in S(g\cup \{C-f-\varepsilon \theta^d\},h)\}\,.
\end{equation}
Note that the quadratic module associated with the constraint set of POP \eqref{eq:compact.POP.form} is Archimedean and so $f^\star_\varepsilon$ can be approximated as closely as desired by the Moment-SOS hierarchy. This approach is similar in spirit to that of \cite{jeyakumar2014polynomial}. However, determining a point $\bar x$ in $S(g,h)$ is not easy in general. 
The hierarchy \eqref{eq:primal2} relying on Putinar-Vasilescu's  Positivstellensatz goes beyond this restriction.

\subsection{Global optimizers}
\label{sec:optimizers}
In this section we introduce a new method to find an approximation of a feasible point of a basic semialgebraic set $S( {g,h})$ as defined in \eqref{eq:S(g,h)}.
We then apply this method to obtain an approximation of a global minimizer $x^\star$ associated to   
$f^\star=\min \{ {f( x):\, x \in S( {g,h})} \}$ via finding a feasible solution of $S( {g\cup \{ {\rho _k^i ( \varepsilon)- f} \},h } )$.
\begin{remark}\label{re:algo.extr.sol}
Let $\varepsilon>0$ be fixed and $k\in\N$ be sufficiently large such that $\rho_k^i( \varepsilon)$ is an upper bound of $f^\star$.  Let $x^\star$ be a global minimizer of $f$ on $S(g,h)$ and
let $\bar x \in S( {g\cup \{ {\rho _k^i( \varepsilon) - f} \},h} )$. 
Then $\bar x \in S( {g,h} )$ and ${f^\star}\le f( {\bar x} ) \le \rho _k^i( \varepsilon) \le {f^\star}+\varepsilon \theta {( {{x^\star}} )^d}$. 
\end{remark}
Let us consider an arbitrary small $\varepsilon>0$. The difference between $\rho_k^i( \varepsilon)$ and $f^\star$ will be as closely as desired to $\varepsilon \theta {( {{x^\star}} )^d}$ for  large enough $k$.  
Assume that the solution set $S( {g\cup \{ {f^\star - f}\},h })$ is finite and denote by $y^\star_\varepsilon$ an optimal solution of SDP \eqref{eq:dual2}. 
In practice, when $k$ is sufficiently large $y^\star_\varepsilon$ satisfies numerically the flat extension condition defined in Section~\ref{sec:Preliminaries}. 
One may then use the algorithm of Henrion and Lasserre  \cite{henrion2005detecting} to extract numerically the support of a representing measure for $y^\star_\varepsilon$ which may include global minimizers of $f^\star=\min \{ {f( x ):\, x \in S( {g,h})}\}$ (see the same extraction in \cite[\S 3.2]{jibetean2005semidefinite}). 
However we cannot guarantee the success of this extraction procedure in theory because the set $S( {g\cup \{ {\rho _k^i( \varepsilon) - f} \},h } )$ may not be zero dimensional when $\rho _k^i( \varepsilon) > f^\star$.
For example, if $f=\|x\|_2^2$ and $S(g,h)=\R^n$, $S( {g\cup \{ {\rho _k^i( \varepsilon) - f} \},h } )$ is a closed ball centered at the origin with radius $\rho _k^i( \varepsilon)^{1/2}$.
The following method aims at overcoming this issue from both theoretical and algorithmic sides.
\paragraph{\textbf{The Adding-Spherical-Constraints method (ASC)}:}
For $a\in\R^n$ and $r\ge 0$, let $\overline {B( {a,r} )}$ (resp. $\partial {B( {a,r} )} $) be the closed ball (resp. sphere) centered at $a$ with radius $r$, i.e.,
\[\overline {B( {a,r} )}  = \{ {x \in {\R^n}:\, \| {x - a} \|_2 \le r} \} \quad\text{(resp.}\quad \partial {B( {a,r} )}  = \{ {x \in {\R^n}:\, \| {x - a} \|_2 = r} \}\text{)}\,.\]
The following result provides an efficient way to find a sequence of additional spherical equality constraints for a given semialgebraic set such that (i) the resulting set is a singleton (i.e., it contains a 
{\em single} real point), and (ii) this point is solution of a non-singular system of {\em linear equations}. 
\begin{lemma}\label{lem:adding.sphere.ine}
Assume that $S( g,h ) \ne \emptyset $. Let ${( {{a_t}} )_{t = 0,\dots,n}} \subset {\R^n}$ such that ${a_t} - {a_0},\, t = 1,\dots,n$ are linear independent in $\R^n$. Let us define the sequence ${( {{\xi_t}} )_{t = 0,1,\dots,n}} \subset {\R_ + }$ as follows:
\begin{equation}\label{eq:radius.poly.sys}
\left\{ \begin{array}{l}
{\xi _0} := \min \{ {{{\| {x - {a_0}} \|}_2^2}:\, x \in S( {g,h} )} \}\,,\\
{\xi _t}: = \min \{ {{{\| {x - {a_t}} \|}_2^2}:\, x \in S( {g,h \cup \{ {{\xi _j} - {{\| {x - {a_j}} \|}_2^2}:\, j = 0,\dots,t - 1} \}} )} \}\,,\\ 
\quad\quad\qquad\qquad\qquad\qquad\qquad\qquad\qquad\qquad\qquad\qquad\qquad\ \ \ \  t = 1,\dots,n\,.
\end{array} \right.
\end{equation}
Then there exists a unique real point $x^\star$ in $S( {g,h \cup \{ {{\xi _t} - {{\| {x - {a_t}} \|}_2^2}:\, t = 0,\dots,n} \}} )$ which satisfies the non-singular linear system of equations
\begin{equation}\label{eq:linear.equation}
{( {{a_t} - {a_0}} )^T}x^\star = -\frac{1}{2}({\xi_t} - {\xi_0} -{\| {{a_t}} \|_2^2} + {\| {{a_0}} \|_2^2}),\ t = 1,\dots,n\,.
\end{equation}
\end{lemma}
The proof of Lemma~\ref{lem:adding.sphere.ine} is postponed to Appendix~\ref{proof:adding.sphere.ine}.
\begin{center}
    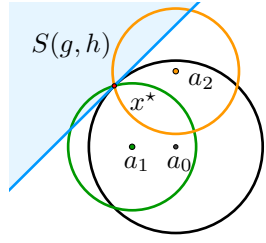
\begin{figure}[htp]
    \begin{center}
\begin{tikzpicture}[scale=\textwidth/30cm,samples=200]
\clip(-5.42,-3.08) rectangle (3.14,4.72);
\fill[line width=1pt,color=qqzzff,fill=qqzzff,fill opacity=0.1] (-6.28,-2.28) -- (-7.08,7.08) -- (2.77,6.77) -- cycle;
\draw [line width=1pt] (0,0) circle (2.8284271247461903cm);
\draw [line width=1pt,color=qqzzqq] (-1.42,0) circle (2.082402458700047cm);
\draw [line width=1pt,color=ffzzqq] (0,2.48) circle (2.0567936211491906cm);
\draw [line width=1pt,color=qqzzff] (-6.28,-2.28)-- (-7.08,7.08);
\draw [line width=1pt,color=qqzzff] (-7.08,7.08)-- (2.77,6.77);
\draw [line width=1pt,color=qqzzff] (2.77,6.77)-- (-6.28,-2.28);
\draw (-0.6,0) node[anchor=north west] {$a_0$};
\draw (-2,0) node[anchor=north west] {$a_1$};
\draw (0.1,2.7) node[anchor=north west] {$a_2$};
\draw (-1.8,2.1) node[anchor=north west] {$x^\star$};
\draw (-5,4.06) node[anchor=north west] {$S(g,h)$};
\begin{scriptsize}
\draw [fill=uuuuuu] (0,0) circle (2pt);
\draw [fill=ffqqqq] (-2,2) circle (2pt);
\draw [fill=qqzzqq] (-1.42,0) circle (2.5pt);
\draw [fill=ffzzqq] (0,2.48) circle (2.5pt);
\draw [fill=xdxdff] (-6.28,-2.28) circle (2.5pt);
\draw [fill=ududff] (-7.08,7.08) circle (2.5pt);
\draw [fill=xdxdff] (2.77,6.77) circle (2.5pt);
\draw[color=xdxdff] (-5.34,4.95) node {$I$};
\draw[color=qqzzff] (-6.28,2.67) node {$i$};
\end{scriptsize}
\end{tikzpicture}
\end{center}
    \caption{Illustration of Lemma \ref{lem:adding.sphere.ine}.}
    \label{fig:adding.spheres}
  \end{figure}
\end{center}
Geometrically speaking, we find a sequence of spheres ${\partial {B( {{a_t},\xi _t^{1/2}})} }$, $t=0,\dots,n$, such that the intersection between these spheres and $S( g,h )$ is the singleton $\{x^\star\}$ (see Figure \ref{fig:adding.spheres}).
Next, we use Lasserre's hierarchy  to compute the optimal values $\xi _t$, $t=0,\dots,n$ of problem \eqref{eq:radius.poly.sys}.
\begin{theorem}\label{theo:poly.sys.ASIM}
Assume that $S(g, h ) \cap \overline {B( {0,L^{1/2}})}  \ne \emptyset $ for some $L>0$. 
Let ${( {{a_t}} )_{t = 0,1,...,n}} \subset {\R^n}$ such that ${a_t}-a_0,\ t = 1,\dots,n,$ are linear independent in $\R^n$. 
Assume that the Moment-SOS hierarchies associated with the following POPs:
\begin{equation}\label{eq:radius.poly.2}
\left\{ \begin{array}{l}
{\xi _0} := \min \{ {{{\| {x - {a_0}} \|}_2^2}:\, x \in S( {g\cup \{ {L - {{\| {x } \|}_2^2}} \},h} ) } \}\,,\\
{\xi _t}: = \min \{ {{\| {x - {a_t}} \|}_2^2}:\, x \in S( g,h\cup {\{ {{\xi _j} - {{\| {x - {a_j}} \|}_2^2}:\, j = 0,\dots,t - 1} \}} ) \} \,,\,\\
\qquad\qquad\qquad\qquad\qquad\qquad\qquad\qquad\qquad\qquad\qquad\qquad \ \ \ \  t = 1,\dots,n\,,
\end{array} \right.
\end{equation}
have finite convergence, and let $w:=\max \{ {u_j,{w_q},1} \}$. For every $k\in \N$, consider the following semidefinite programs:
\begin{equation}\label{eq:approx.raddi.square.0}
\left\{ \begin{array}{l}
\begin{array}{rl}
\eta _{k}^0: = \mathop {\inf }\limits_{y \subset {\R^{s( {2(k+w)} )}}}& {L_y}( {{{\| {x }-a_0 \|}_2^2}} )\\
\text{s.t. }&{M_{k+w-u_j}}( g_jy ) \succeq 0\,,\\
&{M_{k+w - 1}}( {( {L - {{\| {x } \|}_2^2}} )y} ) \succeq 0\,,\\
&{M_{{k+w} - {w_q}}}( {h_qy}) = 0\,,\\
&{y_0} = 1\,,
\end{array}\\
\begin{array}{rl}
\eta _{k}^t: = \mathop {\inf }\limits_{y \subset {\R^{s( {2(k+w)} )}}}& {L_y}( {{{\| {x - {a_t}} \|}_2^2}} )\\
\text{s.t. }&{M_{k+w-u_j}}( g_jy ) \succeq 0\,,\\
&{M_{{k+w} - {w_q}}}( {h_qy} ) = 0\,,\\
&{M_{k+w - 1}}( {( {{\eta _{k}^ j} - {{\| {x - {a_j}} \|}_2^2}} )y} ) = 0\,,\, j = 0,\dots,t - 1\,,\\
&{y_0} = 1\,,
\end{array}\\
\qquad\qquad\qquad\qquad\qquad\qquad\qquad\qquad\qquad\qquad\qquad\qquad  t=1,\dots,n\,.
\end{array} \right.
\end{equation} 
Then there exists $K\in \N$ such that for all $k\ge K$, $\eta _{k}^t = {\xi _t}$, $t = 0,\dots,n$. 
Moreover, there exist $t \in \left\{ {0,\dots,n} \right\}$ and $\tilde K\in \N$ such that for all $k\ge \tilde K$, the solution $y$ of SDP \eqref{eq:approx.raddi.square.0} with value $\eta _{k}^t$ satisfies the flat extension condition, i.e.,
$\text{rank}\left( {{M_{k+w}}\left( y \right)} \right) = \text{rank}\left( {{M_{k}}\left( y \right)} \right)$.
 In addition, $y$ has a representing $\text{rank}\left( {{M_k}\left( y \right)} \right)$-atomic measure $\mu$ and 
${\mathop{\rm supp}\nolimits} \left( \mu  \right) \subset S\left( {g,h} \right)$.
\end{theorem}
The proof of Theorem~\ref{theo:poly.sys.ASIM} is postponed to Appendix~\ref{proof:poly.sys.ASIM}.
\begin{remark}\label{re:generically.ASEM}
In \cite{nie2014optimality} Nie has proved that under the Archimedean assumption, Lasserre's hierarchy has finite convergence generically. 
Thus the conclusion of Theorem \ref{theo:poly.sys.ASIM} is true generically. 
In the final conclusion of Theorem \ref{theo:poly.sys.ASIM}, when $y$ has representing $\rank( {{M_k}( y )} )$-atomic measure $\mu$, we may use the extraction algorithm  \cite{henrion2005detecting} to obtain the atomic support of $\mu$.
\end{remark}
Based on Theorem \ref{theo:poly.sys.ASIM}, Algorithm \ref{alg:ASIM} below finds a feasible point in a nonempty (possibly non-compact) semialgebaric set $S\left( {g,h} \right)$. 
\begin{algorithm}\label{alg:ASIM} PolySys\\
\textbf{Input:} $S( {g,h})\ne \emptyset$, ${( {{a_t}} )_{t = 0,\dots,n}} \subset {\R^n}$ such that ${a_t} - {a_0},\, t = 1,\dots,n$ are linear independent, $\varepsilon>0$ and $k\in\N$.\\ 
\textbf{Output:} $\bar x$.\\
Begin with $t:=0$ and do:
\begin{enumerate}
\item Solve SDP \eqref{eq:primal2} with $f = {\| {x }\|_2^2}$ to obtain $\rho^i_k( \varepsilon)$. Set $L:=\rho^i_k( \varepsilon)$ and go to step 2.
\item Solve SDP \eqref{eq:approx.raddi.square.0} to obtain $\eta _{k}^t$ and an associated solution $y$. 
\begin{enumerate}
\item If $t\le n$ and $\rank\,( {{M_{k+w}}( y )} ) = \rank\,( {{M_k}( y )} )$, i.e., $y$ has a representing measure $\mu$, extract ${\mathop{\rm supp}\nolimits} ( \mu  )$ from $y$ by using the algorithm from \cite{henrion2005detecting}. Take ${\bar x}\in {\mathop{\rm supp}\nolimits} ( \mu  )$ and stop. 
\item If $t\le n$ and $\rank\,( {{M_{k+w}}( y )} ) \ne \rank\,( {{M_k}( y )} )$, set $t:=t+1$ and do again step 2. 
\item If $t= n+1$, stop. 
\end{enumerate}
\end{enumerate}

\end{algorithm}
\begin{proposition}\label{prop:ASEM.terminate}
For $k$ sufficiently large, Algorithm \ref{alg:ASIM} terminates and $\bar x\in S(g,h)$ generically.
\end{proposition}
\begin{proof}
The proof follows from Theorem \ref{theo:poly.sys.ASIM} and Remark \ref{re:generically.ASEM}.
\end{proof}

In Algorithm \ref{alg:ASIM}, step 1 computes the radius $L^{1/2}$ of the ball $\overline {B( {0,L^{1/2}} )}$ which has non-empty intersection with $S({g,h})$. Then step 2 checks the flat extension condition and extracts the solution $\bar x$. 
\begin{remark}
At step 2 in Algorithm \ref{alg:ASIM}, for $k$ sufficiently large, the rank of the moment matrix $\rank( {{M_{k+w}}( y )} )$  decreases to one when $t$ goes from $0$ to $n$. 
Indeed, for each $t$ between $0$ and $n$, we replace the semialgebraic set $S( {g,h} )$ by its intersection with the  $t$ spheres ${\partial {B( {{a_j},\xi _j^{1/2}} )} }$, $j=0,\dots,t-1$. This intersection  includes the support of the measure with moments $y$. 
Since $S( {g,h}) \cap \bigcap_{j = 0}^n {\partial B( {{a_j},\xi _j^{1/2}} )}  = \{ {{x^\star}} \}$,
 this support converges to $\{x^\star\}$ when $t$ goes from $0$ to $n$. 
Thus for large enough $k$, the solution $y$ of SDP \eqref{eq:approx.raddi.square.0} with value $\eta^n_k$ has a representing measure supported on ${x^\star} = ( {{y_{{e_1}}},\dots,{y_{{e_n}}}} )$.  
Here $e_j$, $j=1,\dots,n$ is canonical basis of $\R^n$.
\end{remark}
The decrease of the moment matrix rank in Algorithm \ref{alg:ASIM} for the kissing number problem with $g_1 = x_1^2 + x_2^2 + x_3^2 + x_4^2 -2x_1x_3-2x_2x_4-1$, $h_1 = x_1^2 + x_2^2 - 1$ and $h_2 = x_3^2 + x_4^2 - 1$ is illustrated in Table \ref{tab:rank.decreas}. Here $e_j$, $j=1,\dots,4$ is the canonical basis of $\R^4$. In this example, $\rank( {{M_1}( y )} )$ decreases from $5$ to $1$ when $t$ goes from $0$ to $4$ and  $M_1(y)$ fulfills the flat extension condition at from $t=3$. 
\begin{table}
    \caption{\small Decrease of the moment matrix rank in Algorithm \ref{alg:ASIM}.}
    \label{tab:rank.decreas}
\scriptsize
\begin{center}
\begin{tabular}{|c c| c c c|c|}
\hline 
$t$&$a_t$&$\eta^t_0$& $\rank( {{M_1}( y )} )$& $\rank( {{M_0}( y )} )$ & $\bar x$\\ 
\hline 
$0$&$0_{\R^4}$&$2.0000$&$5$&$1$&$-$\\

$1$&$e_1$&$1.0000$&$3$&$1$&$-$\\
 
$2$&$e_2$&$2.9997$&$3$&$1$&$-$\\

$3$&$e_3$&$1.9998$&$1$&$1$&
(0.9999,

   0.0001,
   
   0.5028,
   
  -0.8611)
\\

$4$&$e_4$&$1.3089$&$1$&$1$&
(1.0000,

   0.0002,
   
   0.4968,
   
   0.8329)\\
\hline 
\end{tabular}    
\end{center}
\end{table}
\begin{remark}
ASC can be used to find an approximation of a real point in $S( {g,h} )$ even if $S( {g,h})$ is positive dimensional. 
This is illustrated later on by our numerical experiments from Section \ref{sec:benchs} (see the polynomial systems corresponding to Id 6, 7, 8 and 13).
\end{remark}
\paragraph{\textbf{Obtaining a minimizer by using the ASC method}:}
We rely on the following algorithm to find the value $\rho_k^i( \varepsilon)$ of SDP \eqref{eq:primal2}, which approximates $f^\star=\min \{ {f( x ):\, x \in S( {g,h} )}\}$, together with an approximation $ \bar x $ of  a minimizer $x^\star$ for this problem. 
\begin{algorithm}\label{alg:obtain.minimizer.algo} PolyOpt\\
\textbf{Input:} $f$, $S(g,h)\ne \emptyset$, $\varepsilon >0$ and $k\in \N$. \\
\textbf{Output:} $\rho_k^i( \varepsilon)$ and $\bar x$.
\begin{enumerate}
\item Solve SDP \eqref{eq:primal2}  to obtain $\rho_k^i( \varepsilon)$.
\item Compute $\bar x$ in $S( {g\cup \{ {{\rho _k^i( \varepsilon)} - f}\},h })$ by using Algorithm \ref{alg:ASIM} and stop. 
\end{enumerate}
\end{algorithm}
\begin{proposition}
If POP $f^\star:=\inf \{ {f( x ):\, x \in S( {g,h} )}\}$ admits an optimal solution at $x^\star$, then for $k$ large enough, Algorithm \ref{alg:obtain.minimizer.algo} terminates  and $f^\star \le \rho_k^i( \varepsilon) \le f^\star + \varepsilon \theta(x^\star)^d$. Moreover, $\bar x \in S( {g,h} )$ and ${f^\star}\le f( {\bar x} )  \le {f^\star}+\varepsilon \theta {( {{x^\star}} )^d}$, generically. 
\end{proposition}
In practice, one performs Algorithm \ref{alg:obtain.minimizer.algo} several times by updating $k:=k+1$ until one obtains $\bar x$ in $S( {g\cup \{ {{\rho _k^i( \varepsilon)} - f}\},h })$. Obviously, one has $f^\star+\varepsilon \theta(x^\star)^d\ge \rho _k^i( \varepsilon)\ge f(\bar x)\ge f^\star$.
\section{Examples}
\label{sec:benchs}
In this section, we report results obtained after  solving some instances of POP \eqref{eq:gen.POP} with Algorithm \ref{alg:obtain.minimizer.algo}. 
As before, let us note $g = \{ {{g_1},\dots,{g_m}} \}$ and $h = \{ {h_1},\dots,{h_l} \}$ the sets of polynomials involved in the inequality constraints and the equality constraints, respectively. 
In particular, the resulting set $S(g,h)$ is unbounded for all examples.

The experiments are performed with both MATLAB R2018a/Yalmip and Julia 1.1.1/JuMP to model the semidefinite optimization problems and  Mosek 8.0 to solve these problems. 
The codes for Algorithm \ref{alg:ASIM} (PolySys) and Algorithm \ref{alg:obtain.minimizer.algo} (PolyOpt) can be downloaded from the link:  https://github.com/maihoanganh.
In these codes, we always set $a_0:=0_{\R^n}$ and $a_1,\dots,a_n$ as  the canonical basis of $\R^n$. 
We use a desktop computer with an Intel(R) Pentium(R) CPU N4200 @ 1.10GHz and 4.00 GB of RAM.    

The input data given in Table \ref{tab:exam.pop} include examples of unconstrained and constrained POPs. 
The corresponding output data, the exact results and timings are given in Table \ref{tab:test.pop}.
In these tables, the SOS hierarchy \eqref{eq:primal2} is solved by optimization models in Yalmip (Y) and JuMP (J). 
The symbol ``$-$" in a column entry indicates that the calculation did not finish in a couple of hours.
\begin{table}
    \caption{\small Examples of POPs.}
    \label{tab:exam.pop}
\scriptsize
\begin{center}
\begin{tabular}{|c| m{2.5cm}| m{8.1cm}|}
\hline 
Id&reference& input data\\ 
\hline 
1&Motzkin& $f=x_1^2x_2^2( {x_1^2 + x_2^2-1} )$\qquad $g=\emptyset$\qquad $h=\emptyset$\\
\hline 
2&Robinson& $f=x_1^6 + x_2^6 - x_1^4x_2^2 - x_1^2x_2^4 - x_1^4 - x_2^4 - x_1^2 - x_2^2 + 3x_1^2x_2^2$

$g=\emptyset$\qquad $h=\emptyset$\\
\hline 
3&Choi-Lam& $f=x_1^4x_2^2+x_2^4+x_1^2-3x_1^2x_2^2$\qquad $g=\emptyset$\qquad $h=\emptyset$\\
\hline 
4&Lax-Lax& $f=x_1x_2x_3-x_1(x_2-x_1)(x_3-x_1)(1-x_1)-x_2(x_1-x_2)(x_3-x_2)(1-x_2)-x_3(x_1-x_3)(x_2-x_3)(1-x_3)-(x_1-1)(x_2-1)(x_3-1)$

$g=\emptyset$\qquad $h=\emptyset$\\
\hline 
5&Delzell& $f=x_1^4x_2^2+x_2^4x_3^2+x_1^2x_3^4-3x_1^2x_2^2x_3^2+x_3^8$
\qquad $g=\emptyset$\qquad $h=\emptyset$\\
\hline 
6& Modified Motzkin &$f=( {x_1^2 + x_2^2-3} )x_1^2x_2^2 $
\qquad $g=\{x_1^2 + x_2^2 - 4\}$
\qquad $h=\emptyset$\\
\hline 
7&\cite[Example 4.3]{hu2016tensor}&$f=x_1^4+x_2^4+x_3^4-4x_1x_3^3$
\qquad $g=\{1-x_1^4+\frac{1}{2}x_2^4-x_3^4\}$
\qquad $h=\emptyset$\\
\hline 
8&\cite[Example 3.1]{jeyakumar2014polynomial}& $f=x_1^2+1$
\qquad $g=\{1- x_2^2, x_2^2-1/4\}$
\qquad $h=\emptyset$\\
\hline 
9&\cite[Example 4.5]{demmel2007representations}& $f=x_1^2+x_2^2 $
\qquad $g=\{x_1^2 -x_1x_2-1, x_1^2 +x_1x_2-1,  x_2^2 -1\}$
\qquad $h=\emptyset$\\
\hline 
10&\cite[Example 4.4]{demmel2007representations}& $f=- \frac{4}{3}x_1^2 + \frac{2}{3}x_2^2 - 2{x_1}{x_2}$
\qquad $g=\{ x_2^2-x_1^2, -x_1x_2 \}$
\qquad $h=\emptyset$\\
\hline 
11& \cite[\S 5.2]{jeyakumar2016solving}& $f=1 + \sum\limits_{j = 2}^{8} ( {{( {{x_j} - x_{j - 1}^2} )}^2} +( {1 - x_j^2} ) ) $

$g=\{x_1,\dots,x_{8}\}$
\qquad $h=\emptyset$\\
\hline
12& \cite[\S 5.3]{jeyakumar2016solving} &
$f=1 + \sum\limits_{l=1}^3 (( x_{2l} - x_{2l-1}^2 )^2 + ( 1 - x_{2l-1} )^2 + 90( x_{2l+2}^2 - x_{2l+1} )^2 + ( x_{2l+1} - 1 )^2 + 10( x_{2l} + x_{2l+2} - 2 )^2 + \frac{1}{10}( x_{2l} - x_{2l+2})^2)$

 $g=\{ x_1,\dots,x_8 \}$
\qquad $h=\emptyset$\\
\hline 
13&\cite[Example A.2]{greuet2014probabilistic}& $f=(x_1^2 + x_2^2 - 2)(x_1^2 + x_2^2)$
\qquad $g=\emptyset$
\qquad $h=\{ (x_1^2 + x_2^2 - 1)(x_1 -3)\}$\\
\hline 
14&\cite[Example A.5]{greuet2014probabilistic} &$f=x_1^6+x_2^6+x_3^6+3x_1^2x_2^2x_3^2-x_1^2(x_2^4+x_3^4)-x_2^2(x_3^4+x_1^4)-x_3^2(x_1^4+x_2^4)$

$g=\emptyset$
\qquad $h=\{x_1+x_2+x_3-1\}$\\
\hline 
15&\cite[Example A.6]{greuet2014probabilistic} &$f=x_1x_2x_3x_4-x_1(x_2-x_1)(x_3-x_1)(x_4-x_1)-x_2(x_1-x_2)(x_3-x2)(x_4-x_2)
    -x_3(x_2-x_3)(x_1-x_3)(x_4-x_3)
    -x_4(x_2-x_4)(x_3-x_4)(x_1-x_4)$

$g=\emptyset$
\qquad $h=\{x_1, x_2-x_3, x_3-x_4\}$\\
\hline 
16&\cite[Example A.4]{greuet2014probabilistic}& $f={( {{x_1} + 1} )^2} + x_2^2$
\qquad $g=\emptyset$ 
\qquad $h=\{ x_1^3 -x_2^2\}$\\
\hline 
17&\cite[Example A.8]{greuet2014probabilistic}& $f=\frac{1}{6}\sum\limits_{j = 1}^{5}(x_j^2+x_{j+5}^2)$
\qquad $g=\emptyset$ 

$h=\{x_6-1,x_{j+6}-x_{j+5}-\frac{1}{6}(x_{j+5}^2-x_j)\,:\,j=1,\dots,4\}$\\
\hline 
18&self made & $f=x_1^6+x_2^2$
\qquad $g=\emptyset$ 
\qquad $h=\{(x_1^2+x_2^2)(1-x_1x_2)^2\}$\\
\hline
19&\cite[Example 2]{gershwin2010kkt} & $f=x_1^2+x_2^2+x_3^2+x_4^2$
\qquad $g=\{\frac{1}{8}-x_4\}$ 

$h=\{x_1+x_2+x_3+x_4-1\}$\\
\hline
20&self made & $f=x_1^3-x_2^2$
\qquad $g=\{x_1,x_2\}$ 
\qquad $h=\{(x_1x_2+1)(x_1-x_2)^2\}$\\
\hline
21& self made& $f=x_1^4-3x_2$
\qquad $g=\{x_1,x_2\}$ 
\qquad $h=\{(x_2-x_1^2)(2x_1^2-x_2)\}$\\
\hline 
22&AM-GM inequality & $f=x_1+x_2+x_3$
\qquad $g=\{x_1,x_2,x_3\}$ 
\qquad $h=\{x_1x_2x_3-1\}$\\
\hline
23&\cite[USSR Olimpiad 1989]{slinko1997ussr} & $f={(x_1+x_2)(x_2+x_3)}$
\qquad $g=\{x_1,x_2,x_3\}$ 

$h=\{x_1x_2x_3(x_1+x_2+x_3)-1\}$\\
\hline
24&\cite[IMO 1990]{djukic2011imo} & $f=x_1(x_1+x_2+x_3)(x_1+x_3+x_4)(x_1+x_2+x_4)+x_2(x_1+x_2+x_3)(x_2+x_3+x_4)(x_1+x_2+x_4)+x_3(x_1+x_2+x_3)(x_1+x_3+x_4)(x_3+x_2+x_4)+x_4(x_4+x_2+x_3)(x_1+x_3+x_4)(x_1+x_2+x_4)-\frac{1}{3}(x_1+x_2+x_3)(x_1+x_3+x_4)(x_1+x_2+x_4)(x_2+x_3+x_4)$

$g=\{x_1,x_2,x_3,x_4\}$ 
\qquad $h=\{x_1x_2+x_2x_3+x_3x_4+x_4x_1-1\}$\\
\hline
25&\cite[IMO 2000]{djukic2011imo} & $f=-(x_1x_2-x_2+1)(x_2x_3-x_3+1)(x_3x_1-x_1+1)$

$g=\{x_1,x_2,x_3\}$ 
\qquad $h=\{x_1x_2x_3-1\}$\\
\hline
\end{tabular}    
\end{center}

\end{table}

\begin{table}
    \caption{\small Numerical experiments with $\varepsilon=10^{-5}$.}
    \label{tab:test.pop}
\scriptsize
\begin{center}
\begin{tabular}{|m{0.1cm}|m{0.7cm}|m{3.5cm}|m{0.1cm}|m{1.3cm}|m{2.4cm}|m{1.2cm}|}
\hline 
Id&$f^\star$ &$\{x^\star\in S(g,h):f(x^\star)=f^\star\}$& $k$&$\rho_k^i(\varepsilon)$& $\bar x$& time (s)\\ 
\hline 
1& $-\frac {1} {27}$&$\{\frac{1}{{\sqrt 3 }}( { \pm 1, \pm 1} )\}$ & $2$&-0.0369&
  (0.5713, 0.5713)& Y: 2.65

J: 2.58\\
\hline 
2& -1&$\{ {( { \pm 1, \pm 1} ),( {0, \pm 1} ),( { \pm 1,0} )} \}$
& $2$& -0.9999&
   (-0.0000, -0.9967)
& Y: 2.90

J: 2.91\\
\hline 
3& 0&$\{ {( { \pm 1, \pm 1} ),( {0, 0} )} \}$
& $1$& 0.0000&
   (-0.0000, -0.0000)
& Y: 3.26

J: 0.22\\
\hline 
4& 0&$\{ (1, 0, 0),(0, 1, 0), (0, 0, 1),$

$ (1, 1, 0), (0, 1, 1), (1, 0, 1) \}$
& $1$& 0.0000&
   (0.00, 0.00, 0.99)
& Y: 4.22

J: 0.30\\
\hline 
5& 0&$\{ {( {0, 0, 0} )} \}$
& $2$& 0.0000&
   (0.00, 0.00, 0.00)
& Y: 65.58

J: 31.84\\
\hline 
6& 0&$\{ {( {0,t} ),\ ( {t,0} )\,:\ | t | \ge 2} \}$& $1$&0.0062&
    (2.0000, 0.0000)
& Y: 4.47

J: 3.10\\
\hline 
7& {\tiny ${1-\sqrt[4]{27}}$}&$\{ {( {0,t,0} )\,:t\in\R} \}\cup \{ {t( {1,0,\sqrt[4]{3}} )\,:|t|\le {1}/{\sqrt[4]{4}}} \}$& $1$&-1.2793&
    (0.70, -0.00, 0.93)
& Y: 3.96

J: 1.94\\
\hline 
8&1&$\{ 0 \} \times (\left[ {\frac{1}{2},1} \right] \cup \left[ { - 1, - \frac{1}{2}} \right])$ & $1$&1.0000&
(0.0000, -0.5081)
& Y: 3.41

J: 2.58\\
\hline 
9&$\frac{{5 + \sqrt 5 }}{2}$&$\{ ( { \pm \frac{{1 + \sqrt 5 }}{2}, \pm 1} ) \}$ &$2$&3.6182&
(1.6181, 1.0000)
& Y: 4.72

J: 2.53\\
\hline 
10&0&$\{( 0,0 )\}$& $5$&
Y: -0.0005

J: -0.0002
&(0.0000, 0.0000)
& Y: 10.08

J: 3.82\\
\hline 
11&1&$\{( {1,\dots,1} )\}$& $0$&Y: $-$

J: 1.0072
&Y: $-$

J: (0.95, 0.96, 0.96, 0.96, 0.97, 0.97, 0.97, 0.97)
& Y: $-$

$\text{J: 2265.71}$\\
\hline
12&1& $\{( {1,\dots,1} )\}$ & $0$&Y: $-$

J: 1.0072
&Y: $-$

J: (0.98, 0.99, 0.99, 0.99, 0.99, 0.99, 0.96, 0.98)
& Y: $-$

$\text{J: 2642.23}$\\
\hline 
13&-1&$\{ {( { t,\pm \sqrt {1 - {t^2}} } ):\, t \in [ { - 1,1} ]} \}$& $2$&
-0.9999
& Y: (1.0000, 0.0000)
   
   J: (0.3663, -0.9304)
& Y: 8.46

J: 3.99\\
\hline 
14&0&$\{\frac{1}{3}( {1,1,1} )\}$&$2$&0.0000&
(0.33, 0.33, 0.33)
& Y: 39.25

J: 6.92\\
\hline 
15&0&$\{( {0,0,0,0})\}$&$1$&0.0000&
  $(0.00,\dots,0.00)$
& Y: 11.06

J: 4.19\\
\hline 
16&1&$\{( 0,0 )\}$&5

10

15

20
&
0.9771

0.9802

0.9857

0.9918
&J: $-$

J: $-$

J: $-$

J: $-$
& J: 1.83

J: 4.49

J: 15.38

J: 154.60
\\
\hline 
17& & &$1$&Y: 1.3216

J: 1.4883&Y: $-$

 J: (1.19, 0.78, 0.46, 0.20, 0.00, 1.0, 0.96, 0.99, 1.08, 1.24)
& $\text{Y: 1846.74}$

J: 639.97 \\
\hline 
18&0 &$\{(0,0)\}$ &$1$&0.0000&(0.0000, 0.0000)
& Y: 2.55

J: 0.34 \\
\hline 
19&$\frac{13}{48}$&$\{( {\frac{7}{24},\frac{7}{24},\frac{7}{24},\frac{1}{8}})\}$&$0$&0.2708&
  $(0.29,0.29,0.29,0.12)$
& Y: 24.44

J: 17.02\\
\hline 
20&$-\frac{4}{27}$&$\{( \frac{2}{ 3},\frac{2}{ 3})\}$&$1$&-0.1487&
 (0.6518, 0.6526)
& Y: 20.52

J: 2.03\\
\hline 
21&$-9$&$\{(\sqrt 3,6)\}$&1&Y: -8.0171

J: -8.5578&
  Y: (1.4172, 4.0172)
  
  J: (1.5280, 4.6701)
& Y: 7.60

J: 2.52\\
\hline 
22&3&$\{(1,1,1)\}$&2&3.0000&(1.00, 1.00, 1.00) 
& Y: 8.57

J: 3.89\\
\hline 
23&2&$\{(1, \sqrt 2 - 1, 1)\}$&5&Y: $-$

J: 2.0000 &Y: $-$

J: (0.99, 0.41, 0.99) 
& Y: $-$

J: 201.29\\
\hline 
24& $\frac{81}{16}$&$\{\frac{1}{2}(1, 1, 1, 1)\}$& 0& 5.0625&(0.49, 0.49, 0.49, 0.49) 
& Y: 11.37

J: 2.18\\
\hline 
25& -1&$\{(1, 1, 1)\}$& 1&-0.9974 &(1.0, 1.0, 1.0) 
& Y: 35.33

J: 7.45\\
\hline 
\end{tabular}    
\end{center}

\end{table}

Id 1-5 are unconstrained POPs. 
Id 6-12 are POPs with inequality constrains, Id 13-18 are POPs with equality constraints and Id 19-25 are POPs with both inequality and equality constraints. 
Id 8, 11 and 12 correspond to examples from Jeyakumar et al. \cite{jeyakumar2014polynomial, jeyakumar2016solving}. Id 9 and 10 are selected from Demmel et al. \cite{demmel2007representations}. Id 13-17 come from Greuet et al. \cite{greuet2014probabilistic}. Id 23, 24 and 25 are POPs constructed from some inequalities issued from Mathematics competitions mentioned in \cite{slinko1997ussr, djukic2011imo}, yielding non-compact POPs with known optimal values and optimizers.

Even though the sets of minimizers associated to Id 6, 7, 8 and 13 are positive dimensional, we can still extract an approximate one of them by using our ASC algorithm. Note that ASC computes a real point $\bar x$ in $S( {g\cup \{ {{\rho _k^i( \varepsilon)} - f}\},h })$ which is an outer approximation of $\{x^\star\in S(g,h): f(x^\star)=f^\star\}$ for $k$ sufficiently large.

In Table \ref{tab:test.pop}, Algorithm \ref{alg:obtain.minimizer.algo} terminates at some order $k\le 5$ for all POPs except Id 16. Note that for Id 16, the global minimum does not satisfy the KKT conditions, and the convergence rate of $(\rho _k^i( \varepsilon ))_{k\in\N}$ is very poor when $\varepsilon \le 10^{-5}$. 
We  overcome this issue by fixing $k$, multiplying $\varepsilon$ by 10, and solving again the relaxations. 
The computational cost that we must pay here is due to the largest gap $\varepsilon\, \theta {( {{x^\star}} )^d}$ between  $\rho_k^i( \varepsilon)$ and $f^\star$. 
This behavior is illustrated in Table \ref{tab:fixk.updateeps}.
\begin{table}
    \caption{\small Numerical experiments for Id 16 with various values of $\varepsilon$.}
    \label{tab:fixk.updateeps}
\scriptsize
\begin{center}
\begin{tabular}{|m{0.6cm}|m{1.3cm}|m{2cm}|m{1.2cm}|}
\hline 
$\varepsilon$&$\rho_5^i(\varepsilon)$& $\bar x$& time (s)\\ 
\hline 
$10^{-4}$&Y: 0.9492

J: 0.9778&J: $-$

Y: $-$
&Y: 4.47

J: 2.02\\
 \hline
$10^{-3}$&Y: 0.9528

J: 0.9783&J: $-$

Y: $-$
&Y:4.62

J: 1.59\\
 \hline
 $10^{-2}$&Y: 0.9550

J: 0.9884&J: $-$

Y: $-$
&Y: 4.45

J: 1.29\\
 \hline
$10^{-1}$&Y: 1.0479

J: 1.0774&(0.0000, 0.0000)&Y: 17.64

J: 3.76\\
 \hline
\end{tabular}    
\end{center}
\end{table}

In Id 18, even if the ideal $\langle h \rangle $ is not radical and $V( h )$ is not equidimensional (the assumptions required to apply the framework in \cite{greuet2014probabilistic} are not guaranteed) our ASC method can still extract one solution of the problem.

For Id 21, we can improve the quality of the approximation $\rho_k^i( \varepsilon)$ of the optimal value $f^\star$ by fixing $k=1$, dividing $\varepsilon$ by 10, and solving again the relaxations. 
This is illustrated in Table \ref{tab:fixk.updateeps2}.
\begin{table}
    \caption{\small Numerical experiments for Id 21 with various values of $\varepsilon$.}
    \label{tab:fixk.updateeps2}
\scriptsize
\begin{center}
\begin{tabular}{|m{0.6cm}|m{1.3cm}|m{2.4cm}|m{1.2cm}|}
\hline 
$\varepsilon$ &$\rho_1^i(\varepsilon)$& $\bar x$& time (s)\\ 
\hline 
$10^{-6}$&Y: -8.2078

J: -8.9392&
  Y: $(1.4525,4.2199)$
  
  J: $(1.6593, 5.5069)$
& Y: 8.46

J: 3.10\\
\hline 
$10^{-7}$&Y: -8.4889

J: -8.9935&
  Y: $(1.5116,4.5701)$
  
  J: $(1.7086, 5.8387)$
& Y: 8.34

J: 2.61\\
\hline 
$10^{-8}$&Y: -8.4915

J: -8.9968&
  Y: $(1.5122,4.5739)$
  
  J: $(1.7158, 5.8873)$
& Y: 8.37

J: 2.70\\
\hline 
$10^{-9}$&Y: -7.9335

J: -8.9949&
  Y: $(1.4026,3.9346)$
  
  J: $(1.7113, 5.8572)$
& Y: 8.30

J: 2.54\\
\hline 
\end{tabular}    
\end{center}
\end{table}

We emphasize that we can  customize the $\varepsilon$ parameter for different purposes.  On the one hand, one increases $\varepsilon$ to improve the convergence speed of the sequence $(\rho^i_k(\varepsilon))_{k\in\N}$ to the neighborhood $[f^\star,f^\star+\varepsilon \theta(x^\star)^d]$ of $f^\star$ (see Table \ref{tab:fixk.updateeps}). On the other hand, one decreases $\varepsilon$ to improve the accuracy of the approximate optimal value $\rho^i_k(\varepsilon)$ and the approximate optimal solution $\bar x$ (see Table \ref{tab:fixk.updateeps2}).

Our numerical benchmarks also show that modeling in JuMP is faster and provides more accurate outputs than modeling in Yalmip. 
In particular, the JuMP implementation is the only one which provides solutions for Id 11, 12, 17 and 23.

Let us now denote by $k_\varepsilon$ the smallest nonnegative integer such that $\rho_{k_\varepsilon}^i(\varepsilon)\ge f^\star$, for each $\varepsilon>0$. 
The graph of the function $\varepsilon^{-1}\mapsto k_\varepsilon$ on $(0,100]$ for Id 9 and Id 16 is illustrated in Figure \ref{fig:plot.complex}. 
Here Id 9 (resp. Id 16) is an example of POP such that the global minimums satisfy the KKT condition (resp. do not satisfy the KKT condition). 
We can experimentally compare the complexity of Algorithm \ref{alg:obtain.minimizer.algo} in both cases. 
For Id 9, the function seems to increase as slowly as a constant function, which is in deep contrast with Id 16, where the function increases more quickly and seems to have a step-wise linear growth. 
Theorem \ref{theo:constr.theo} states that $k_\varepsilon$ has an upper bound which is exponential in $\varepsilon^{-1}$, i.e., $k_\varepsilon\le C \exp \left( O(\varepsilon^{-1})^C \right) -d$ as $\varepsilon^{-1}  \to  \infty$ if $0_{\R^n}\in S(g,h)$.  
This is an open question to guarantee that $k_\varepsilon\le O(\varepsilon^{-N})$  as $\varepsilon \downarrow 0$ for some $N>0$ in the constrained case. 
Another open questions are whether and how the KKT condition affects the convergence rate of $(\rho^i_k(\varepsilon))_{k}$.

\begin{center}
    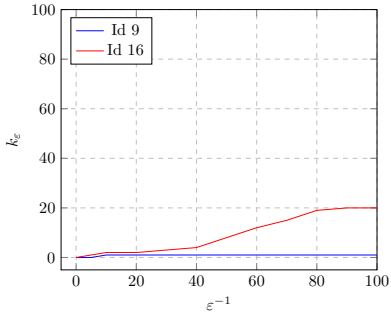
\begin{figure}[htp]
    \begin{center}
\begin{tikzpicture}[scale=\textwidth/20cm,samples=200]
\begin{axis}[
    xlabel={$\varepsilon^{-1}$},
    ylabel={$k_\varepsilon$},
    xmin=-5, xmax=100,
    ymin=-5, ymax=100,
    xtick={0,20,40,60,80,100},
    ytick={0,20,40,60,80,100},
    legend pos=north west,
    ymajorgrids=true,
    xmajorgrids=true,
    grid style=dashed,
]
 
\addplot[
    color=blue
    ]
    coordinates {
    (0,0)(5,0)(10,1)(20,1)(30,1)(40,1)(50,1)(60,1)(70,1)(80,1)(90,1)(100,1)
    };
    
\addplot[
    color=red,
    ]
    coordinates {
    (0,0)(5,1)(10,2)(20,2)(30,3)(40,4)(50,8)(60,12)(70,15)(80,19)(90,20)(100,20)
    };
    
    \legend{Id 9, Id 16}
 
\end{axis}
\end{tikzpicture}
\end{center}
    \caption{Plot of the complexity.}
    \label{fig:plot.complex}
  \end{figure}
\end{center}

\section{Conclusion}
In this paper, we have established new proofs for two representations of globally nonnegative polynomials and polynomials nonnegative on semialgebraic sets based on the homogeneous representations in \cite{reznick1995uniform} and \cite{putinar1999solving}. 
Then we rely on these representations to convert them into a practical numerical scheme for approximating the global minimum.
We provide  converging hierarchies of semidefinite relaxations for unconstrained and constrained polynomial optimization problems.
We have also introduced a method based on adding spherical constraints (ASC) to solve systems of polynomial equalities and inequalities, and to obtain global solutions of polynomial optimization problems as well. 

In view of the practical efficiency of ASC, a topic of further investigation is to provide a more detailed comparison with other methods for solving polynomial systems. 
Another direction of research is to derive a sparse variant of Putinar-Vasilescu's Positivstellensatz in order to improve the scalability of our optimization framework. 
For this, we cannot directly rely on existing sparse polynomial optimization techniques by~\cite{waki2008algorithm} since the perturbation term is not sparse.
One possible workaround would be to first  derive a sparse variant of Reznick's Positivstellensatz, involving sparse uniform denominators.


\paragraph{\textbf{Acknowledgements}.} This research was started at LAAS-CNRS in the MAC team, where the first author did his Master's  internship. 
We would like to thank all members of the MAC team for their support. 
We also thank all members from the review board of the Master's defense of the first author in Universit\'e de Limoges for their useful comments on this paper.
The second author was supported by the European Research Council (ERC) under the European's Union Horizon 2020 research and innovation program (grant agreement 666981 TAMING).
The third author was supported by the FMJH Program PGMO (EPICS project) and  EDF, Thales, Orange et Criteo, as well as from the Tremplin ERC Stg Grant ANR-18-ERC2-0004-01 (T-COPS project).
\appendix
\section{Appendix}
\label{sec:Appendix}
\subsection{Proof of Theorem \ref{theo:representation.nonnegative.poly}} 
\label{proof:representation.nonnegative.poly}
\begin{proof}
Let $\tilde f = x_{n + 1}^{2d} f ( x/x_{n + 1} )$ be the degree $2d$ homogenization of $f$. Since $f$ is globally nonnegative, $\tilde f$ is nonnegative on $\R^{n+1}$. Let $\varepsilon>0$ be fixed. We claim that 
\[\tilde f + \varepsilon{\| {( {x,{x_{n + 1}}} )} \|_2^{2d}} \in \R[ {x,{x_{n + 1}}}]\]
is positive definite, i.e., is homogeneous and positive on ${\R^{n + 1}}\backslash \{ 0_{{\R^{n+1}}} \}$. 
Since  
\[{\| {( {x,{x_{n + 1}}} )} \|_2^{2 d}} = {( {x_1^2 + \dots + x_n^2 + x_{n + 1}^2} )^{d}}\,,\]
the polynomial ${\| {( {x,{x_{n + 1}}} )} \|_2^{2 d}}$ is homogeneous of degree $2 d$ on $\R^{n+1}$. 
From this and since $\tilde f$ is homogeneous of degree $2 d$, $\tilde f + \varepsilon{\| {( {x,{x_{n + 1}}} )} \|_2^{2 d}}$ is homogeneous of degree $2 d$. 
For every $( {x,{x_{n + 1}}} )\in {\R^{n + 1}}\backslash \{ 0_{{\R^{n+1}}} \}$, $\| {( {x,{x_{n + 1}}} )} \|_2 > 0$. 
From this and since $\tilde f$ is nonnegative on $\R^{n+1}$, 
\[\tilde f( {x,{x_{n + 1}}} ) + \varepsilon {\| {( {x,{x_{n + 1}}} )} \|_2^{2 d}} > 0 \,,\]
for all $( {x,{x_{n + 1}}})\in{\R^{n + 1}}\backslash \{ 0_{{\R^{n+1}}} \}$.
In addition, it is not hard to show that 
\[\inf\{\tilde f(x,{x_{n + 1}})+ \varepsilon {\| {( {x,{x_{n + 1}}} )} \|_2^{2 d}}\,:(x,{x_{n + 1}})\in \mathbb S^{n}\}\ge \varepsilon\]
and
\[\begin{array}{rl}
&\sup\{\tilde f(x,{x_{n + 1}})+ \varepsilon {\| {( {x,{x_{n + 1}}} )} \|_2^{2 d}}\,:(x,{x_{n + 1}})\in \mathbb S^{n}\}\\
\le &\sup\{\tilde f(x,{x_{n + 1}})\,:(x,{x_{n + 1}})\in \mathbb S^{n}\}+\varepsilon\\
\le &\| \tilde f\|_{1}+\varepsilon \,.
\end{array}\]
Thus, $\delta(\tilde f +\|.\|^{2d}_2)\le (\| \tilde f\|_{1}+\varepsilon)/\varepsilon=(\|  f\|_{1}+\varepsilon)/\varepsilon$.
From this and by applying Lemma \ref{lem:homogeneous.sos} with $p=\tilde f+ \varepsilon {\| {( {x,{x_{n + 1}}} )} \|_2^{2 d}}$, for $k_\varepsilon\in \N$ and 
\[k_\varepsilon\ge \frac{{2(n+1)d(2d - 1)}}{{(4\log 2)}}(\varepsilon^{-1}\|  f\|_{1}+1) - \frac{{n + 1+2d}}{2}\,,\]
there exists ${{\tilde \sigma }_\varepsilon } \in \Sigma [ {x,{x_{n + 1}}} ]_{k_{\varepsilon} + d}$ such that
\[{\| {( {x,{x_{n + 1}}} )} \|_2^{2{k_\varepsilon }}}( {\tilde f + \varepsilon {{\| ( {x,{x_{n + 1}}} )\|}_2^{2 d}}}) = {{\tilde \sigma }_\varepsilon }\,.\]
By replacing $x_{n+1}$ by $1$, one has
\[{\theta^{{k_\varepsilon }}}( {f + \varepsilon {{\theta}^{d}}} ) = {{\tilde \sigma }_\varepsilon }( {x,1} )\,.\]
Let us note ${\sigma _\varepsilon }( x ) := {{\tilde \sigma }_\varepsilon }( {x,1})$, for every $x\in \R^n$. 
Since ${{\tilde \sigma }_\varepsilon } \in \Sigma [ {x,{x_{n + 1}}} ]_{k_{\varepsilon} + d}$, it follows that ${\sigma _\varepsilon } \in \Sigma [ x ]_{k_{\varepsilon} + d}$, yielding the desired result.
\end{proof}

\subsection{Proof of Theorem \ref{theo:Representation.nonnegative2}}
\label{proof:Representation.nonnegative2}
Jacobi \cite[Theorem 7]{jacobi2001representation} proves another result of Putinar and Vasilescu \cite[Theorem 4.2]{putinar1999solving}, which states that if $f,g_1,\dots,g_m$ are polynomials of even degree and $\tilde f>0$ on $S(\{\tilde g_1,\dots,\tilde g_m\})\backslash \{0\}$, then $\theta^kf\in Q(g)$ for $k$ large enough, where $\tilde h$ is the homogenization of given polynomial $h$. 
The idea of Jacobi is to apply Putinar's Positivstellensatz for the fact that $\tilde f>0$ on the intersection of $S(\{\tilde g_1,\dots,\tilde g_m\})$ with the unit sphere. 
Our proof for Theorem \ref{theo:Representation.nonnegative2} replaces $\tilde f$ here by the perturbation of $\tilde f$ and replaces the unit sphere by a sphere with an arbitrary radius $L^{1/2}$.
 This changing will support the computation of the complexity in Proposition \ref {prop:Complexity.Putinar-Vasilescu} later on.
\begin{proof}
1. Let us prove the conclusion under condition (i).
For every $h \in \R[ x ]$, we define by $\hat h$ the degree $2d_1 ( h )$ homogenization of $h$, i.e.,
\begin{equation}
\label{eq:homogenous.formula2}
\hat h( {x,{x_{n + 1}}} ) = x_{n + 1}^{2d_1 ( h )}h( {\displaystyle{x}/{{{x_{n + 1}}}}} )\,.
\end{equation}
where ${d_1}( h ): = {1 + \lfloor {{{\deg ( h )}}/{2}} \rfloor }$.
Let $\varepsilon  > 0$ be fixed. 
Let $\tilde f$ be the degree $2d$ homogenization of $f$. 
Then $\tilde f + \varepsilon {\| {( {x,{x_{n + 1}}} )} \|_2^{2d}}$ is a homogeneous polynomial of degree $2d$. 
Let $\hat g := \{ {{\hat g_1},\dots,{\hat g_m}} \}$ and $L>0$. We will show that 
\begin{equation}\label{eq:positive.semialgebraic.homogeneous2}
\tilde f + \varepsilon {\| {( {x,{x_{n + 1}}} )}\|_2^{2d}} \ge \varepsilon L^{d}\text{ on }S( \hat g, \{L-\| {( {x,{x_{n + 1}}} )}\|_2^2\})\,.
\end{equation}
Let $( {y,{y_{n + 1}}}) \in S( \hat g, \{L-\| ( x,x_{n + 1} \|_2^2\})$ be fixed. 
By \eqref{eq:homogenous.formula2}, one has
\begin{equation}\label{eq:inequality.homogenous2}
y_{n + 1}^{2d_1 ( g_j )}{g_j}( {\displaystyle{y}/{{{y_{n + 1}}}}} )={{\hat g}_j}( {y,{y_{n + 1}}} ) \ge 0,\,j = 1,\dots,m\,.
\end{equation}
We consider the following two cases:
\begin{itemize}
\item Case 1: ${y_{n + 1}} \ne 0$. 
For $j=1,\dots,m$, by \eqref{eq:inequality.homogenous2} and since $y_{n + 1}^{2d_1 \left( g_j \right)}>0$, ${g_j}( {\displaystyle{y}/{{{y_{n + 1}}}}} ) \ge 0$.
It implies that $\displaystyle{y}/{{{y_{n + 1}}}} \in S\left( g\right)$. 
By assumption, $f( \displaystyle{y}/{{{y_{n + 1}}}}) \ge 0$. 
From this, $\tilde f( {y,{y_{n + 1}}}) = y_{n + 1}^{2d}f( {{y}/{{{y_{n + 1}}}}} ) \ge 0$.
 From this and since $\|( {y,{y_{n + 1}}} )\|_2^2=L$, $\tilde f( {y,{y_{n + 1}}}) + \varepsilon {\| {( {y,{y_{n + 1}}} )} \|_2^{2d}} \geq \varepsilon L^{d}$.
\item Case 2: ${y_{n + 1}} = 0$. By definition of $d_1(f)$, $x_{n+1}$ divides $\tilde f$ so $\tilde f( {y,{y_{n + 1}}}) = 0$. 
From this and since $\|( {y,{y_{n + 1}}})\|_2^2=L$, 
\[\tilde f( {y,{y_{n + 1}}}) + \varepsilon {\| {( {y,{y_{n + 1}}} )} \|_2^{2d}} = \varepsilon L^{d}\,.\]
\end{itemize}
Thus, \eqref{eq:positive.semialgebraic.homogeneous2} holds. 
It is not hard to show that $S( \hat g \cup \{L-\| {( {x,{x_{n + 1}}} )}\|_2^2\})$ satisfies the Archimedean condition. 
From this and by applying Theorem \ref{theo:Putinar's.Positivstellensatz2} (i), 
\[\tilde f + \varepsilon {\| {( {x,{x_{n + 1}}} )} \|^{2d}} \in Q( \hat g,{\{ { L - {{\| {( {x,{x_{n + 1}}} )} \|}^2_2}} \}} )\,.\]
Then there exist $\psi  _j  \in \Sigma [ {x,{x_{n + 1}}} ],\ j = 0,1,\dots,m$ and $\varphi  \in \R [ {x,{x_{n + 1}}} ]$ such that
\[\tilde f + \varepsilon {\| {( {x,{x_{n + 1}}} )} \|_2^{2d}} = \psi _0  + \sum\limits_{j = 1}^m {\psi _j{{\hat g}_j}}   + ( {L - {{\| {( {x,{x_{n + 1}}} )} \|}_2^2}} )\varphi \,.\]
Let $( {z,{z_{n + 1}}}) \in {\R^{n + 1}}\backslash \{ 0\}$. By replacing ${( {x,{x_{n + 1}}} )}$ in the last equality by $L^{1/2}\frac{{( {z,{z_{n + 1}}} )}}{{\| {( {z,{z_{n + 1}}} )} \|_2}}$ and the fact that ${\tilde f + \varepsilon {\| {( {x,{x_{n + 1}}})} \|_2^{2d}}}, \ \hat g_1,\dots,\hat g_m$ are homogeneous polynomials of degree $2d,\ 2d_1 \left( g_1 \right),\dots,2d_1 \left( g_m\right)$ respectively, one has
\[\begin{array}{rl}
&\left( {\tilde f( {z,{z_{n + 1}}} ) + \varepsilon {{\| {( {z,{z_{n + 1}}} )} \|}_2^{2d}}} \right)L^d{\| {( {z,{z_{n + 1}}} )} \|_2^{ - 2d}}\\
 = & \psi _0 \left( L^{1/2}\frac{( z,z_{n + 1})}{\| ( z,z_{n + 1} ) \|_2} \right)\\
& + \sum\limits_{j = 1}^m {\psi _j \left( L^{1/2}\frac{( z,z_{n + 1} )}{\| ( z,z_{n + 1}) \|} \right)\hat g_j( z,z_{n + 1} )L^{d_1(g_j)}\| ( z,z_{n + 1} ) \|_2^{ - 2d_1 ( g_j )}} \,.
\end{array}\]
Set 
\[K: = \max \{ {2d,\deg (\psi _0  ),2d_1( g_1) + \deg ( \psi _1),\dots,2d_1( g_m) + \deg ( \psi _m),2+2\lceil\deg(\varphi)/2\rceil}\}\,.\]
Obviously, $K$ is even.
After multiplying the two sides of the last equality with $\| ( z,z_{n + 1}) \|^K$, one has
\[\begin{array}{rl}
&\left( {\tilde f( {z,{z_{n + 1}}} ) + \varepsilon {{\| {( {z,{z_{n + 1}}} )} \|}_2^{2d}}} \right)L^d{\| {( {z,{z_{n + 1}}} )} \|_2^{K - 2d}}\\
 =&   \bar \psi _0( z,z_{n + 1} )  + \sum\limits_{j = 1}^m {L^{d_1(g_j)} \bar \psi _j (z,z_{n + 1} ) \hat g_j( z,z_{n + 1} )} \,,
\end{array}\]
where 
\begin{align*}\bar \psi _0 &:=\psi _0 \left( L^{1/2}\frac{( x,x_{n + 1})}{\| ( x,x_{n + 1} ) \|_2} \right)\| ( x,x_{n + 1} ) \|_2^{ K}\text\,,\\
\bar \psi_j &:=\psi _j \left( L^{1/2}\frac{( x,x_{n + 1} )}{\| ( x,x_{n + 1}) \|_2} \right)\| ( x,x_{n + 1} ) \|_2^{ K- 2d_1 ( g_j )}\,,\,j=1,\dots,m\,.
\end{align*}
Since $( {z,{z_{n + 1}}}) \in {\R^{n + 1}}\backslash \{ 0\}$ is arbitrary,
\begin{equation}\label{eq:equal.homoge} 
\left( {\tilde f + \varepsilon {{\| {( {x,{x_{n + 1}}} )} \|}_2^{2d}}} \right)L^d{\| {( {x,{x_{n + 1}}} )} \|_2^{K - 2d}}=  \bar \psi _0 + \sum\limits_{j = 1}^m {L^{d_1(g_j)}\bar \psi _j  \hat g_j} \,.
\end{equation}
Let $j\in\{0,\dots,m\}$ be fixed and set $g_0:=1$ and $d_1(g_0):=0$. We will show that 
\begin{equation}\label{eq:form.psi}
\bar \psi_j=s_j+\xi_j \|(x,x_{n+1})\|_2\,,
\end{equation} 
for some $s_j\in \Sigma[x,x_{n+1}]_{K/2-d_1(g_j)}$ and $\xi_j\in \R[x,x_{n+1}]$. 
By definition of $K$ ($K- 2d_1 ( g_j )\ge \deg(\psi_j)$, $j=1,\dots,m$) and since $\psi_j$ is SOS, $\bar \psi_j=\sum\limits_{t = 1}^{l} {\phi_{t}^2}$ with $\phi_{t}:=h_t+p_t\|(x,x_{n+1})\|_2$ for some $h_t\in\R[x,x_n]_{K/2-d_1(g_j)}$ and $p_t\in\R[x,x_n]_{K/2-d_1(g_j)-1}$. Then 
\[\bar \psi_j=\sum\limits_{t = 1}^{l} {(h_t^2+p_t^2\|(x,x_{n+1})\|_2^2)}+2\|(x,x_{n+1})\|_2\sum\limits_{t = 1}^{l} {h_tp_t}\,.\]
By setting $s_j:=\sum\limits_{t = 1}^{l} {(h_t^2+p_t^2\|(x,x_{n+1})\|_2^2)}$ and $\xi_j:=2\sum\limits_{t = 1}^{l} {h_tp_t}$, \eqref{eq:form.psi} follows. By \eqref{eq:equal.homoge} and \eqref{eq:form.psi},
\[\begin{array}{rl}
&\left( {\tilde f + \varepsilon {{\| {( {x,{x_{n + 1}}} )} \|}_2^{2d}}} \right)L^d{\| {( {x,{x_{n + 1}}} )} \|_2^{K - 2d}}\\
=&  s_0 + \sum\limits_{j = 1}^m {L^{d_1(g_j)}s_j  \hat g_j}+\| ( x,x_{n + 1} ) \|_2\sum\limits_{j = 1}^m {L^{d_1(g_j)}\xi_j  \hat g_j} \,.
\end{array}\]
Note that $\| {( {x,{x_{n + 1}}} )} \|_2$ is not polynomial. From the last equality and since the right hand side is polynomial, the left hand side must be a polynomial, so $\sum\limits_{j = 1}^m {L^{d_1(g_j)}\xi_j  \hat g_j}=0$. Thus,
\[\left( {\tilde f + \varepsilon {{\| {( {x,{x_{n + 1}}} )} \|}_2^{2d}}} \right)L^d{\| {( {x,{x_{n + 1}}} )} \|_2^{K - 2d}}=  s_0 + \sum\limits_{j = 1}^m {L^{d_1(g_j)}s_j  \hat g_j} \,.\]
By setting $k_\varepsilon:=K/2-d$, ${s_0} \in \Sigma {[ x,x_{n+1}  ]_{{k_\varepsilon } + d}}$, ${s _j} \in \Sigma {[ x,x_{n+1} ]_{{k_\varepsilon } + d - d_1( {{g_j}})}},\,j = 1,\dots,m$ and 
$L^d{\| {( {x,{x_{n + 1}}})} \|_2^{2{k_\varepsilon }}}( {\tilde f + \varepsilon {{\| {( {x,{x_{n + 1}}} )} \|}_2^{2d}}} ) = {s _0} + \sum\limits_{j = 1}^m {L^{d_1(g_j)}{s_j}{{\hat g}_j}} $.
By replacing $x_{n+1}$ by $1$,
$
L^d{\theta^{{k_\varepsilon }}}( {f + \varepsilon {\theta^{d}}} ) = {s _0}( {x,1} ) + \sum\limits_{j = 1}^m {L^{d_1(g_j)}{s_j}( {x,1} ){g_j}} $.
Since ${s_0} \in \Sigma {[ x,x_{n+1}]_{{k_\varepsilon } + d}}$, ${s_j} \in \Sigma {[ x,x_{n+1} ]_{{k_\varepsilon } + d - d_1( {{g_j}} )}},\,j = 1,\dots,m$, ${s _0}( {x,1} ) \in \Sigma {[ x ]_{{k_\varepsilon } + d}}$, ${s _j}( {x,1}) \in \Sigma {[ x]_{{k_\varepsilon } + d - d_1( {{g_j}})}},\,j = 1,\dots,m$. 
Note that $d_1(g_j)\ge d_2(g_j)=u_j$, $j=1,\dots,m$.
Hence,  ${\theta ^{{k_\varepsilon }}}( {f + \varepsilon {\theta ^{d}}} ) \in  {Q_{{k_\varepsilon } + d}}( g )$ by the definition of truncated quadratic module.

\smallskip

\noindent 2. Let us show the conclusion under condition (ii).
We do a similar process as part 1 (under condition (i)). 
The difference is that $\hat h$ here is defined as the degree $2d_2 ( h )$ (instead of $2d_1 ( h )$) homogenization of $h \in \R[ x ]$ and the proof for \eqref{eq:positive.semialgebraic.homogeneous2}. To show \eqref{eq:positive.semialgebraic.homogeneous2} in Case 2: $y_{n+1}=0$ here, we rely on the constraint $g_m=f+\lambda$ for some $\lambda \ge 0$ (instead of $d \ge d_1(f)$ and that $x_{n+1}$ divides $\tilde f$). More explicitly, we assume by contradiction that 
\[\tilde f({y,{y_{n + 1}}}) + \varepsilon {\| {( {y,{y_{n + 1}}} )}\|_2^{2d}} <\varepsilon L^{d}.\]
 From this and since $0 \le {\hat g_m}(y,{y_{n + 1}}) = \hat f(y,0)$,
\[\begin{array}{rl}
\varepsilon \|({y,{y_{n + 1}}})\|_2^{2d} &\le 0^{2d-2d_2(f)}\hat f(y,0) + \varepsilon \|({y,{y_{n + 1}}})\|_2^{2d}\\
&= y_{n+1}^{2d-2d_2(f)}\hat f(y,y_{n+1}) + \varepsilon \|({y,{y_{n + 1}}})\|_2^{2d} \\
&= \tilde f({y,{y_{n + 1}}}) + \varepsilon {\| {( {y,{y_{n + 1}}} )}\|_2^{2d}} <\varepsilon L^{d}\,.
\end{array}\]
It follows that $\|({y,{y_{n + 1}}})\|_2^2<L=\|({y,{y_{n + 1}}})\|_2^2$. 
It is impossible. 
Thus, $\tilde f({y,{y_{n + 1}}}) + \varepsilon {\| {( {y,{y_{n + 1}}} )}\|_2^{2d}} \ge \varepsilon L^{d}$.
\end{proof}
\subsection{Proof of Proposition~\ref{prop:Complexity.Putinar-Vasilescu}}
\label{proof:Complexity.Putinar-Vasilescu}
\begin{proof}
We keep all the notation from the proof of Theorem \ref{theo:Representation.nonnegative2}. Without loss of generality, let us assume that  $L=1/4$. 
By \eqref{eq:positive.semialgebraic.homogeneous2}, $S( \hat g,{\{ { L - {{\| {( {x,{x_{n + 1}}} )} \|}^2_2}} \}} )\subset (-1,1)^{n+1}$ and
\begin{equation}\label{eq:first.ineq}
\min\{\tilde f + \varepsilon {\| {( {x,{x_{n + 1}}} )} \|_2^{2d}}\,:( {x,{x_{n + 1}}} )\in S( \hat g,{\{ { L - {{\| {( {x,{x_{n + 1}}} )} \|}^2_2}} \}} )\}\ge 2^{-2d}\varepsilon\,.
\end{equation}
By definition of $\hat g$ and since $0_{\R^n}\in S(g)$, $(0_{\R^n},L^{1/2})\in S( \hat g,{\{ { L - {{\| {( {x,{x_{n + 1}}} )} \|}^2_2}} \}} )$.
Thus, $S( \hat g,{\{ { L - {{\| {( {x,{x_{n + 1}}} )} \|}^2_2}} \}} )$ is nonempty.
By definition of $K$,
\[\tilde f + \varepsilon {\| {( {x,{x_{n + 1}}} )} \|_2^{2d}} \in Q_{K/2}( \hat g,{\{ { L - {{\| {( {x,{x_{n + 1}}} )} \|}^2_2}} \}} )\,.\]
Since $\| {(x,{x_{n + 1}})} \|_2^{2d} = \sum\limits_{\bar \alpha  \in \N_d^{n + 1}} {c_{n+1}(\bar \alpha)(x,{x_{n + 1}})^{2\bar \alpha}} $, one has
\[\begin{array}{rl}
\| {\tilde f + \varepsilon \| {(x,{x_{n + 1}})} \|_2^{2d}} \|_{\max}&\le \|{\tilde f } \|_{\max} +\varepsilon \| {\| {(x,{x_{n + 1}})} \|_2^{2d}} \|_{\max}\\
&=\|{\tilde f } \|_{\max} +\varepsilon \max \{ \displaystyle {c_{n+1}(\bar \alpha)}/{c_{n+1}(2\bar \alpha)}\,:{\bar \alpha \in \N^{n+1}_d}\}\,.
\end{array}\]
Note that $\tilde f = \sum\nolimits_\alpha  {f_\alpha x^\alpha x_{n+1}^{2d-|\alpha|}}$, so $\| \tilde f  \|_{\max}=\| f  \|_{\max ,d}$. Thus, 
\begin{equation}\label{eq:second.ineq}
\| {\tilde f + \varepsilon \| {(x,{x_{n + 1}})} \|_2^{2d}} \|_{\max}\le \|{ f } \|_{\max, d} +\varepsilon \max \{ \displaystyle {c_{n+1}(\bar \alpha)}/{c_{n+1}(2\bar \alpha)}\,:{\bar \alpha \in \N^{n+1}_d}\}\,.
\end{equation}
From these and using Theorem \ref{theo:Complexity.Putinar},  one can choose 
\[\frac{K}{2}\ge C\exp \left( \left( 4^{d+1}d^2(n + 1)^{2d}\left(\varepsilon^{-1} \| f  \|_{\max, d}+ \max \left\{ \displaystyle \frac{c_{n+1}(\bar \alpha)}{c_{n+1}(2\bar \alpha)}\,:{\bar \alpha \in \N^{n+1}_d}\right\}\right) \right)^C \right)\,,\]
for some $C>0$. The right hand side of this inequality comes from \eqref{eq:first.ineq}, \eqref{eq:second.ineq} and the fact that the function $t \mapsto c\exp \left( {b{t^c}} \right)$ with positive constants $b$ and $c$ is increasing on $[0,\infty)$.
By setting $k_\varepsilon=K/2-d$, the conclusion follows.
\end{proof}
\subsection{Proof of Proposition \ref{prop:strong.duality.con}}
\label{proof:strong.duality.con}
\begin{proof}
We denote by $P$ and $P^\star$ feasible set and optimal solutions set for the moment SDP \eqref{eq:dual-sdp}, respectively.
We claim that $P$ is nonempty. Indeed, with $z$ being the moment sequence of the $1$-atomic measure $\theta(\bar x)^{-1}\delta_{\bar x}$ for some $\bar x\in S(g)$, it is not hard to check that the truncation of $z$ is a feasible solution of \eqref{eq:dual-sdp}. By setting $C: = {L_z}({\theta ^k}(f + \varepsilon {\theta ^d}))$, $P^\star$ is also the set of optimal solutions for
\begin{equation}\label{eq:dual-sdp2}
\begin{array}{rl}
{\tau_k^2}(\varepsilon) = \inf &{L_y}( {{\theta ^k}( {f +  \varepsilon {\theta ^d}} )} )\\
\text{s.t.}&y = {(y_\alpha )_{\alpha  \in \N^n_{2( {d + k})}}} \subset \R\,,\\
&{L_y}( {{\theta ^k}( {f +  \varepsilon {\theta ^d}} )} )\le C\,,\\
&{M_{k + d - {u_j}}}( {{g_j}y}) \succeq 0 \,,\;j = 0,\dots,m\,,\\
&{L_y}( {{\theta ^k}}) = 1\,,
\end{array}
\end{equation}
with $g_0=1$.
Denote by $\bar P$ the feasible set of \eqref{eq:dual-sdp2}.
Note that $\bar P$ is nonempty since the truncation of $z$ is also a feasible solution of \eqref{eq:dual-sdp2}. 
We will prove that $\bar P$ is bounded. 
By definition of $\theta$, ${\theta ^r} = \sum\nolimits_{\alpha  \in \N_r^n} {{C_\alpha^r }{x^{2\alpha }}}$ for all $r\in\N$, where $C_\alpha^r\ge 1$ for all $\alpha  \in \N_r^n$.
Since $f-f^\star \ge 0$ on $S(g)$, by Theorem \ref{theo:Representation.nonnegative2}, there exists $K\in \N$ such that
\[{\theta ^K}\left( {f  - f^\star +  \displaystyle\frac \varepsilon 2 {\theta ^d}} \right) \in Q_{K + d}{( g )}\,.\]
Note that $K$ depends only on $f$, $g$ and $\varepsilon$. 
Assume that $k\ge K$. 
Since $\theta^{k-K} \in \Sigma {[ x ]_{k-K}}$, one has
\[{\theta ^k}\left( {f - {f^\star}+  \displaystyle\frac \varepsilon 2 {\theta ^d} }\right) = \theta^{k - K} {\theta ^{K}}  \left( {f  - {f^\star}+  \displaystyle\frac \varepsilon 2 {\theta ^d}} \right) \in Q _{k + d}{( g)}\,.\] 
Then there exist $G_j\succeq 0$, $j=1,\dots,m$ such that
\[{\theta ^k}\left(f - {f^ \star } + \frac{\varepsilon }{2}{\theta ^d}\right) = \sum\limits_{j = 0}^m {v_{k + d - {u_j}}^T{G_j}{v_{k + d - {u_j}}}{g_j}}  = \sum\limits_{j = 0}^m {\trace({G_j}{v_{d + k - {u_j}}}v_{d + k - {u_j}}^Tg_j)} \,.\]
Let $y\in\bar P$. From these and since ${M_{k + d - {u_j}}}( {{g_j}y})\succeq 0$,
\[{L_y}\left( {{\theta ^k}\left( {f - {f^ \star } + \frac{\varepsilon }{2}{\theta ^d}} \right)} \right) = \sum\limits_{j = 0}^m {\trace({G_j}{M_{d + k - {u_j}}}({g_j}y))}  \ge 0\,.\]
Thus, for every $\beta \in \N_{k + d}^n$,
\[\begin{array}{rl}
\displaystyle\frac \varepsilon 2 {y_{2\beta }} &\le \displaystyle\frac \varepsilon 2 C_\beta ^{k + d}{y_{2\beta }} \le \displaystyle\frac \varepsilon 2 \sum\limits_{\alpha  \in \N_{k + d}^n} {C_\alpha ^{k + d}{y_{2\alpha }}}  = \displaystyle\frac \varepsilon 2 \sum\limits_{\alpha  \in \N_{k + d}^n} {C_\alpha ^{k + d}{L_y}({x^{2\alpha }})} \\
& = \displaystyle\frac \varepsilon 2 {L_y}({\theta ^{k + d}})  \le \displaystyle\frac \varepsilon 2 {L_y}({\theta ^{k + d}}) + {L_y}\left({\theta ^k}\left(f -  f^\star +\displaystyle\frac \varepsilon 2 \theta ^d\right)\right) \\
&= {L_y}({\theta ^k}(f + \varepsilon {\theta ^d})) -  f^\star {L_y}({\theta ^k})\le C - f^\star\,, 
\end{array}\]
since every element $y_{2\alpha}$ on the diagonal of the positive semidefinite matrix ${M_{k + d}}( {y})$ is nonnegative. 
 Thus, ${y_{2\beta }}\le 2(C - f^\star)\varepsilon^{-1}$ for every $\beta \in \N_{k + d}^n$. 
Since ${M_{k + d }}( {y})\succeq 0$, $| {{y_{\alpha  + \beta }}} | \le {y_{2\beta }} \le 2(C - \underline f)\varepsilon^{-1}$ for all $\alpha,\beta \in \N_{k + d}^n$. 
This implies that $\|y\|_2$ is bounded by $2(C - \underline f)\varepsilon^{-1}\sqrt {s(2(d + k))}$. 
Since the objective function of \eqref{eq:dual-sdp2} is linear and the feasible set of \eqref{eq:dual-sdp2} is closed and bounded, the set $P^\star$ of optimal solutions of \eqref{eq:dual-sdp2} is nonempty and bounded. 
By using Trnovska's result \cite[Corrollary 1]{trnovska2005strong},  $\rho_k^2(\varepsilon)=\tau_k^2(\varepsilon)$, yielding the desired conclusion.

\end{proof}
\subsection{Proof of Proposition \ref{prop:strong.duality.con2}}
\label{proof:strong.duality.con2}
\begin{proof}
We do a similar process as the proof of Proposition \ref{prop:strong.duality.con}. 
The difference is the bound of $y$ here  obtained by the inequality constraint $g_m=f-\underline f\ge 0$ in stead of Positivstellensatz. 
Thus, we do not need $k$ sufficient large here. More explicitly, one has $\|y\|_2 \le (C - \underline f)\varepsilon^{-1}\sqrt {s(2(d + k))}$, since for every $\beta \in \N_{k + d}^n$,
\[\begin{array}{rl}
\varepsilon {y_{2\beta }} &\le \varepsilon C_\beta ^{k + d}{y_{2\beta }} \le \varepsilon \sum\limits_{\alpha  \in \N_{k + d}^n} {C_\alpha ^{k + d}{y_{2\alpha }}}  = \varepsilon \sum\limits_{\alpha  \in \N_{k + d}^n} {C_\alpha ^{k + d}{L_y}({x^{2\alpha }})} \\
 &\le \varepsilon \sum\limits_{\alpha  \in \N_{k + d}^n} {C_\alpha ^{k + d}{L_y}({x^{2\alpha }})}  + \sum\limits_{\alpha  \in \N_k^n} {C_\alpha ^k{L_y}({x^{2\alpha }}{g_m})} \\
& = \varepsilon {L_y}({\theta ^{k + d}}) + {L_y}({\theta ^k}(f - \underline f )) = {L_y}({\theta ^k}(f + \varepsilon {\theta ^d})) - \underline f {L_y}({\theta ^k})\\
 &\le C - \underline f\,.  
\end{array}\]
\end{proof}
\subsection{Proof of Lemma \ref{lem:adding.sphere.ine}}
\label{proof:adding.sphere.ine}
\begin{proof}
Let us show that $(\xi_t)_{t=0,\dots,n}$ is real sequence. Obviously, $\xi_t\ge 0$, $t=0,\dots,n$. Set $r_t=\xi_t^{1/2}$, $t=0,\dots,n$. 
Then
\begin{equation}\label{radius.poly.sys}
\left\{ \begin{array}{l}
{r_0} = d( {{a_0},S( g,h )} )\ge 0\,,\\
{r_t} = d( {{a_t},S( g,h ) \cap \partial {B( {{a_0},{r_0}} )}  \cap \dots \cap \partial {B( {{a_{t - 1}},{r_{t - 1}}} )} } )\ge 0\,,\, t = 1,\dots,n\,,
\end{array} \right.
\end{equation}
where $d(a,A):=\inf\{\|x-a\|_2:\, x\in A\}$ for $a\in\R^n$ and $A\subset \R^n$.
It is sufficient to prove that $r_t$ is real, $t=0,\dots,n$.
It is easy to see that $S( g,h )$ is closed and  $S( g,h)$ is nonempty by assumption. From this, ${r_0}$ is nonnegative real and $S( g,h ) \cap \partial {B( {{a_0},{r_0}} )} $ is also closed and nonempty. It implies that $r_1$ is nonnegative real and  
\[{S( g,h) \cap \partial {B( {{a_0},{r_0}} )}  \cap \partial {B( {{a_{1}},{r_{1}}} )} }\]
 is also closed and nonempty. By induction, for $t \in \{ {0,\dots,n} \}$, $r_t$ is nonnegative real and
\[{S( g,h ) \cap \partial {B( {{a_0},{r_0}} )}  \cap \dots \cap \partial {B( {{a_{t}},{r_{t}}} )} }\]
is closed and nonempty. 
Thus,
\[S( {g,h \cup \{ {{\xi _t} - {{ \| {x - {a_t}} \|}_2^2}:\, t = 0,\dots,n} \}} )=S( g,h ) \cap \partial {B( {{a_0},{r_0}})}  \cap \dots \cap \partial {B( {{a_n},{r_n}} )} \ne\emptyset\,.\]
Let ${x^\star} \in S( {g,h \cup \{ {{\xi _t} - {{ \| {x - {a_t}} \|}_2^2}:\, t = 0,\dots,n} \}} )$.
Then
${x^\star} \in \partial B\left( {{a_0},{r_0}} \right) \cap \dots \cap \partial B( {{a_n},{r_n}} )$.
It follows that ${\| {{x^\star} - {a_t}} \|_2^2} = r_t^2=\xi_t,\, t = 0,\dots,n$. For $t = 1,\dots,n$,
\[\xi_t - \xi_0 = {\| {{x^\star} - {a_t}} \|_2^2} - {\| {{x^\star} - {a_0}} \|_2^2} = -2{( {{a_t} - {a_0}} )^T}{x^\star} + {\| {{a_t}} \|_2^2} - {\| {{a_0}} \|_2^2}\,.\]
It implies \eqref{eq:linear.equation}. Denote 
\begin{equation*}
A = \left( {\begin{array}{*{20}{c}}
{{{( {{a_1} - {a_0}} )}^T}}\\
{\dots}\\
{{{( {{a_n} - {a_0}} )}^T}}
\end{array}} \right)\qquad\text{ and }\qquad b = -\frac{1}{2}\left( {\begin{array}{*{20}{c}}
{\xi_1 - \xi_0 - {{\| {{a_1}} \|}_2^2} + {{\| {{a_0}} \|}_2^2}}\\
{\dots}\\
{\xi_n - \xi_0 - {{\| {{a_n}} \|}_2^2} + {{\| {{a_0}} \|}_2^2}}
\end{array}} \right).
\end{equation*}
The system  \eqref{eq:linear.equation} can be rewritten as $Ax^\star=b$. Since ${a_j} - {a_0},\, j = 1,\dots,n$ are linear independent in $\R^n$, $A$ is invertible.  Hence, $x^\star$ is determined uniquely by $x^\star=A^{-1}b$.
\end{proof}
\subsection{Proof of Theorem \ref{theo:poly.sys.ASIM}}
\label{proof:poly.sys.ASIM}
\begin{proof}
We will prove by induction that for $t \in \left\{ {0,\dots,n} \right\}$,
\begin{equation}\label{eq:need.to.show.ind}
\exists {K_t} \in \N:\,\forall k \ge {K_t}\,,\,\forall j \in \{ {0,\dots,t} \}\,,\,  \eta _k^j = {\xi _j}\,.
\end{equation}
For $t=0$, \eqref{eq:approx.raddi.square.0} is the SDP relaxation of order $k+w$ of 
\[{\xi _0} = \min \{ {{{\| {x - {a_0}} \|}_2^2}:\, x \in S( {g\cup \{ {L - {{\| {x } \|}_2^2}} \},h} ) } \}\,.\]
By assumption, ${( {\eta _k^0} )_{k \in \N}}$ finitely converges to $\xi_0$, i.e., there exist $K_0\in \N$ such that $\eta _k^0 = {\xi _0}$ for all $k\ge K_0$. 
It follows that \eqref{eq:need.to.show.ind} is true for $t=0$. 
Assume that \eqref{eq:need.to.show.ind} is true for $t=T$, i.e.,
\begin{equation}\label{eq:assump.induc}
\exists {K_T} \in \N:\,\forall k \ge {K_T}\,,\,\forall j \in \{ {0,\dots,T} \}\,,\, \eta _k^j = {\xi _j}\,.
\end{equation}
We will show that \eqref{eq:need.to.show.ind} is true for $t=T+1$. 
By \eqref{eq:approx.raddi.square.0} and \eqref{eq:assump.induc}, for all $k\ge K_T$,
\[\begin{array}{rl}
\eta _k^{T+1} = \mathop {\inf }\limits_{y \subset {\R^{s( {2(k+w)} )}}}& {L_y}( {{{\| {x - {a_{T+1}}} \|}_2^2}} )\\
\text{s.t. }&{M_{k+w - {u_j}}}( {{g_j}y} ) \succeq 0\,,\\
&{M_{k+w - {w_q}}}( {h_qy} ) = 0\,,\\
&{M_{k+w - 1}}( {( {{\xi_ j} - {{\| {x - {a_j}} \|}_2^2}})y}) = 0\,,\, j = 0,\dots,T\,,\\
&{y_0} = 1
\end{array}\]
is the SDP relaxation of order $k+w$ of problem 
\[{\xi _{T+1}} = \min \{ {{{\| {x - {a_t}} \|}_2^2}:\, x \in S( {g,h \cup \{ {{\xi _j} - {{\| {x- {a_j}} \|}_2^2},\, j = 0,\dots,T} \}} )} \}\,.\]
By assumption, ${( {\eta _k^{T+1}})_{k \ge K_T}}$ finitely converges to $\xi_{T+1}$, i.e., there exist $K_{T+1}\ge K_T$ such that $\eta_k^{T+1} = {\xi _{T+1}}$ for all $k\ge K_{T+1}$. 
It follows that \eqref{eq:assump.induc} is true for $t=T+1$. 
Thus, \eqref{eq:assump.induc} is true for $t=0,\dots,n$. 
For $t=n$, there exists $K=K_n\in \N$ such that for all $k\ge K$, $\eta_k^t = {\xi _t}$, $t = 0,\dots,n$. Let $k\ge K$ be fixed. Let $y$ be the solution of problem 
\[\begin{array}{rl}
\eta _k^n: = \mathop {\inf }\limits_{y \subset {\R^{s( {2(k+w)})}}}& {L_y}( {{{\| {x - {a_n}} \|}_2^2}} )\\
\text{s.t. }&{M_{k+w - {u_j}}}( {{g_j}y} ) \succeq 0\,,\\
&{M_{k+w - {w_q}}}( {h_qy}) = 0\,,\\
&{M_{k+w - 1}}( {( {{\xi_ j} - {{\| {x - {a_j}} \|}_2^2}} )y} ) \succeq 0\,,\, j = 0,\dots,n-1\,,\\
&{y_0} = 1\,,
\end{array}\]
which is the SDP hierarchy relaxation of order $k+w$ of problem 
\[{\xi _{n}} = \min \{ {{{\| {x - {a_n}} \|}_2^2}:\, x \in S( {g,h \cup \{ {{\xi _j} - {{\| {x - {a_j}} \|}_2^2}:\, j = 0,\dots,n-1} \}} )} \}\,.\]
We will prove that this latter problem has a unique minimizer $x^\star$. Set $\hat h:=h\cup \{ \xi _0-\|x-a_0\|_2^2\}$. Then 
\[\left\{ \begin{array}{l}
{\xi _0} = \min \{ {{{\| {x - {a_0}} \|}_2^2}:\, x \in S( g,\hat h)} \}\,,\\
{\xi _t} = \min \{ {{{\| {x - {a_t}} \|}_2^2}:\, x \in S( { g,\hat h \cup \{ {{\xi _j} - {{\| {x - {a_j}} \|}_2^2}:\, j = 0,\dots,t - 1} \}} )} \}\,,\\
\qquad\qquad\qquad\qquad\qquad\qquad\qquad\qquad\qquad\qquad\qquad\qquad  \ \ \ \  \ \ t = 1,\dots,n\,.
\end{array} \right.\]
By Lemma \ref{lem:adding.sphere.ine} (with $S(g,h):=S(g,\hat h)$), there exists $x^\star$ such that 
\begin{equation}\label{eq:intersect.spheres4}
\begin{array}{rl}
&S( {g,h \cup \{ {{\xi _j} - {{\| {x - {a_j}} \|}_2^2}:\,j = 0,\dots,n } \}} ) \\
= &S( { g, \hat h \cup \{ {{\xi _j} - {{\| {x - {a_j}} \|}_2^2}:\,j = 0,\dots,n } \}} ) = \{ {{x^\star}} \}\,.
\end{array}
\end{equation}
Let $ a$ be a minimizer of the above POP with value $\xi_n$. Then 
\[a\in S( {g,h \cup \{ {{\xi _j} - {{\| {x - {a_j}} \|}_2^2}:\, j = 0,\dots,n-1} \}} )\,,\]
 and ${\| {a - {a_n}} \|}_2^2=\xi_n$. It follows that
\[a\in  S( {g,h \cup \{ {{\xi _j} - {{\| {x - {a_j}} \|}_2^2}:\,j = 0,\dots,n } \}} )\,.\]
From this and by \eqref{eq:intersect.spheres4}, $a=x^\star$. 
Since the above POP with optimal value $\xi _n$ has a unique minimizer $x^\star$ and its Lasserre's hierarchy has finite convergence, the solution $y$ of the SDP with optimal value $\eta_k^n$ must have a representing $1$-atomic measure $\mu$ supported on $x^\star$. Then $y$ satisfies the flat extension condition when $k$ is large enough and ${\mathop{\rm supp}\nolimits}( \mu  ) = \{ {{x^\star}}\}\subset S(g,h)$. The conclusion follows.
\end{proof}
%
\bibliographystyle{abbrv}

\end{document}